\documentclass[11pt,fullpage]{article}

\usepackage{fullpage}
\usepackage{amsthm}
\usepackage{amsmath}
\usepackage{amsfonts}
\usepackage{rotating}
\usepackage{authblk}

\newtheorem{theorem}{Theorem}
\newtheorem{lemma}[theorem]{Lemma}
\newtheorem{corollary}[theorem]{Corollary}

\title{\bf The Hamiltonian Circuit Polytope}

\author[1]{Latife Gen\c{c}-Kaya}
\author[2]{J. N. Hooker}
\affil[1]{Groupon Inc., \texttt{latife@gmail.com}}
\affil[2]{Carnegie Mellon University, \texttt{jh38@andrew.cmu.edu}}

\date{Revised December 2018}

\begin{document}

\maketitle

\begin{abstract}
The hamiltonian circuit polytope is the convex hull of feasible solutions for the \mbox{circuit} constraint, which provides a succinct formulation of the traveling salesman and other sequencing problems.  We study the polytope by establishing its \mbox{dimension}, developing tools for the identification of facets, and using these tools to derive several families of facets.  The tools \mbox{include} necessary and sufficient conditions for an inequality to be facet defining, and an algorithm for generating all undominated circuits.  We use a novel approach to identifying families of facet-defining inequalities, based on the structure of variable indices rather than on subgraphs such as combs or subtours.  This leads to our main result, a hierarchy of families of facet-defining inequalities and polynomial-time separation algorithms for them.  
\end{abstract}



%

\section{Introduction}

The {\em circuit constraint} \cite{Lau78,CasLab97,ShuBer94} requires that a sequence of vertices in
a directed graph define a hamiltonian circuit.  Given a directed graph $G$ on vertices $1, \ldots, n$, the constraint is written
\begin{equation}
\mbox{circuit$(x_1, \ldots, x_n)$} \label{eq:cir}
\end{equation}
where variable $x_i$ denote the vertex that follows vertex $i$ in the
sequence.  The constraint
requires that $x=(x_1, \ldots, x_n)$ describe a hamiltonian
circuit of $G$.  For brevity, we will say that an $x$ satisfying
(\ref{eq:cir}) is a {\em circuit}.

We define the {\em hamiltonian circuit polytope} to be the convex hull of the feasible
solutions of (\ref{eq:cir}) when $G$ is a complete graph.  Thus if the {\em domain} $D_i$ of variable $x_i$ is the
set of values $x_i$ can take, we suppose that each $D_i=\{1, \ldots, n\}$.  To our
knowledge, this polytope has not been studied.  Our objective is to establish its basic properties and provide tools for identifying classes of facets of the polytope.  We use these tools to describe several families of facets.  In particular, we identify a hierarchy of families of facets, along with polynomial-time separation algorithms.

A circuit should be distinguished from a permutation.  Although a circuit $x=(x_1, \ldots, x_n)$ is always a permutation of $(1, \ldots, n)$, a permutation is not necessarily a circuit.  For example, $(x_1,x_2,x_3,x_4)=(3,4,2,1)$ is a circuit that goes from 1 to 3 to 2 to 4, and back to 1.  However, the permutation $(x_1,x_2,x_3,x_4)=(3,4,1,2)$ is not a circuit because it contains two subtours (1 to 3 to 1, 2 to 4 to 2).  If the domain of each $x_i$ is $\{1, \ldots, n\}$, then $n!$
values of $x$ are permutations but only $(n-1)!$ of these are
circuits.  

The convex hull of permutations of $1, \ldots, n$ is the {\em permutohedron}, which has been studied for at least a century \cite{Sch1911}.  The permutohedron is well understood and quite different from the hamiltonian circuit polytope, although we will see that they have some facets in common.

The paper is organized as follows.  We begin by clarifying the connection between the circuit constraint and the traveling salesman problem, and how facets identified here can provide lower bounds for the problem.  We then introduce general variable domains and establish the dimension of the hamiltonian circuit polytope for an arbitrary domain.  Following this, we develop two tools for identifying facets of the polytope: (a) necessary and sufficient conditions for an inequality with at most $n-4$ variables to be facet-defining, stated in terms of {\em undominated} circuits; and (b) a simple greedy algorithm that generates all undominated circuits, along with a proof of its completeness.

We then apply these tools to analyze the structure of the hamiltonian circuit polytope.  A key element of the analysis is a novel approach to identifying families of facets.  Rather than associate facet-defining inequalities with graphical substructures such as combs and subtours, we associate them with the position of their variables in the sequence $x_1, \ldots, x_n$.  Different patterns of variable indices give rise to different classes of facets.  

We first describe a family of inequalities that are facet defining for both the permutohedron and the hamiltonian circuit polytope, and we provide an exhaustive list of two-term facets.  We then proceed to our main result, which is a hierarchy of facets of increasing combinatorial complexity.  We explicitly describe the facets on levels~0, 1 and~2 of the hierarchy and show how similar analysis can identify facets on higher levels.  We conclude by presenting polynomial-time separation algorithms for all families of facets identified here.  The algorithms yield a separating cut for each family whenever one exists.

\section{Sequencing Problems}

The circuit constraint is useful for formulating combinatorial
problems that involve permutations or sequencing.  One of the best
known such problems is the {\em traveling salesman problem} (TSP), which may be
very succinctly written
\begin{equation}
\begin{array}{l}
{\displaystyle \min \; \sum_{i=1}^n c_{ix_i}
} \vspace{1ex} \\
\mbox{circuit}(x_1, \ldots, x_n), \;\; x_i\in D_i, \; i=1, \ldots n
\vspace{.5ex}
\end{array} \label{eq:tsp}
\end{equation}
where $c_{ij}$ is the distance from city $i$ to city $j$.  The
objective is to visit each city once, and return to the starting
city, in such a way as to minimize the total travel distance.

The facet-defining inequalities we obtain for the hamiltonian circuit polytope can be used to obtain lower bounds on the optimal value of the TSP and related problems.  Bounds of this sort can be indispensable for solving the problem.  In addition, domain filtering methods developed elsewhere \cite{CasLab97,KayHoo06,ShuBer94} for the circuit constraint can be useful for eliminating infeasible values from the variable domains.  

Bounds are normally obtained for the TSP by formulating it with \mbox{0--1} variables $y_{ij}$, where $y_{ij}=1$ if vertex $j$
immediately follows vertex $i$ in the hamiltonian circuit.  The
problem (\ref{eq:tsp}) can then be written
\begin{equation}
\begin{array}{l}
{\displaystyle \min \;\sum_{ij} c_{ij}y_{ij}
} \vspace{1ex} \\
{\displaystyle \sum_j y_{ij} = \sum_j y_{ji} = 1, \;\; i=1, \ldots,
n
} \vspace{1ex} \\
{\displaystyle \sum_{\scriptsize
\begin{array}{@{}c@{}}
i\in V \\
j\not\in V
\end{array}
} \hspace{-.5ex} y_{ij} \geq 1, \;\; \mbox{all $V\subset \{1,
\ldots, n\}$ with $2\leq |V|\leq n-2$}
} \vspace{.5ex} \\
y_{ij} \in \{0,1\}, \;\;\mbox{all $i,j$} 
\end{array} \label{eq:tsp2}
\end{equation}
The polyhedral structure of problem (\ref{eq:tsp2}) has been
intensively analyzed, and surveys of this work may be found in
\cite{BalFis02,JunReiRin95,Nad02}.  Bounds are obtained by solving a linear programming problem that minimizes the objective function in (\ref{eq:tsp2}) subject to valid inequalities for this problem, including facet-defining inequalities.  

Although the objective function of model (\ref{eq:tsp}) is nonlinear, valid inequalities for (\ref{eq:tsp}) can be mapped into the \mbox{0--1} model (\ref{eq:tsp2}), where the \mbox{objective} function is linear.  This is accomplished by the simple change of variable $x_i=\sum_j jy_{ij}$, which transforms linear inequalities in the variables $x_i$ into linear inequalities in the \mbox{0--1} variables $y_{ij}$.  These can be combined with valid inequalities that have been developed for the \mbox{0--1} model, so as to obtain a lower bound on the objective function value.

This strategy is applied in \cite{BerHoo12,BerHoo13} to graph coloring problems.  Facet-defining inequalities for a formulation in terms of finite-domain variables $x_i$ are transformed into valid inequalities for the standard \mbox{0--1} model.  The resulting cuts are quite different from known classes of valid inequalities.  They yield tighter bounds in substantially less compututation time.  

We leave to future research the question of how the valid inequalities obtained here compare with known valid cuts when mapped into the \mbox{0--1} model.  Our focus is on the structure of the hamiltonian circuit polytope, which is an interesting object of study in its own right.

The {\em all-different constraint} \cite{Lau78,Reg94} provides a third formulation for the TSP, which may be written
\begin{equation}
\begin{array}{l}
{\displaystyle \min \; \sum_{i=1}^n c_{x_i x_{i+1}}
} \vspace{1ex} \\
\mbox{all-different}(x_1, \ldots, x_n), \;\;\; x_i\in\{1, \ldots, n\}, \; i=1,\ldots, n
\vspace{.5ex}
\end{array} \label{eq:tsp3}
\end{equation}
where $x_{n+1}$ is identified with $x_1$.  The all-different constraint simply requires that $x_1,\ldots, x_n$ be a permutation of $1, \ldots, n$, and the convex hull of its solutions is the permutohedron.  Although the facets of the permutohedron are well known (see Section~\ref{permutation}), they cannot be transformed into linear inequalities for the \mbox{0--1} model (\ref{eq:tsp2}) because the variables $x_i$ have a different meaning than in the circuit model (\ref{eq:tsp}).  In addition, missing edges in the graph $G$ cannot be represented by removing elements from the domains $D_i$ as in (\ref{eq:tsp}).

\section{General Domains}

A peculiar characteristic of the circuit constraint is that the
values of its variables are indices of other variables.  Because the
vertex immediately after $x_i$ is $x_{x_i}$, the value of $x_i$ must
index a variable.  The numbers $1, \ldots, n$ are normally used as
indices, but this is an arbitrary choice.  One could just as well
use any other set of distinct numbers, which would give rise to a
different polytope.  Thus the hamiltonian circuit polytope cannot be
fully understood unless it is characterized for general numerical
domains, and not just for $1, \ldots, n$.  

We therefore generalize the circuit constraint so that each domain
$D_i$ is drawn from an arbitrary set $\{v_1, \ldots, v_n\}$ of
nonnegative real numbers.  The constraint is written
\begin{equation}
\mbox{circuit}(x_{v_1}, \ldots, x_{v_n}) \label{eq:cir2}
\end{equation}
It is convenient to assume $v_1<\cdots< v_n$.  Thus
circuit$(x_0,x_{2.3},x_{3.1})$ is a well-formed circuit constraint
if the variable domains are subsets of $\{0,2.3,3.1\}$.  The
nonnegativity of the $v_i$s does not sacrifice generality when the domains are finite, since one
can always translate the origin so that the feasible points lie in
the nonnegative orthant.

Most of the results stated here are valid for a general finite domain.  However, to simplify notation we develop the facets in the hierarchy mentioned earlier only for $\{1, \ldots, n\}$.

To avoid an additional layer of subscripts, we will consistently
abuse notation by writing $x_{v_i}$ as $x_i$.  We therefore write
the constraint (\ref{eq:cir2}) as (\ref{eq:cir}), with the understanding that $x=(x_1, \ldots, x_n)$ satisfies (\ref{eq:cir}) if and
only if $\pi_1, \ldots, \pi_n$ is a permutation of $1, \ldots,
n$, where $\pi_1=1$ and $v_{\pi_i} = x_{\pi_{i-1}}$ for $i=2,
\ldots n$.

We define the hamiltionian circuit polytope $H_n(v)$ with respect to $v=(v_1,
\ldots, v_n)$ to be the convex hull of the feasible solutions of
(\ref{eq:cir}) for full domains; that is, each domain $D_i$ is
$\{v_1, \ldots, v_n\}$.  All of the facet-defining inequalities
we identify for full domains are valid inequalities for
smaller domains, even if they may not define facets of the convex
hull.

\section{Dimension of the Polytope}

We begin by establishing the dimension of the hamiltonian circuit polytope.

\begin{theorem}  \label{th:dimension}
The dimension of $H_n(v)$ is $n-2$ for $n=2,3$
and $n-1$ for $n \geq 4$.
\end{theorem}

\begin{proof}
The polytope $H_n(v)$ is a point $(v_2,v_1)$ for $n=2$
and the line segment from $(v_2,v_3,v_1)$ to $(v_3,v_1,v_2)$ for
$n=3$.  In either case the dimension is $n-2$.

To prove the theorem for $n\geq 4$, note first that all feasible
points for (\ref{eq:cir}) satisfy
\begin{equation}
\sum_{i=1}^n x_i = \sum_{i=1}^n v_i \label{eq:affine}
\end{equation}
(Recall that $x_i$ is shorthand for $x_{v_i}$.)  Thus, $H_n(v)$ has
dimension at most $n-1$.  To show it has dimension exactly $n-1$, it
suffices to exhibit $n$ affinely independent points in $H_n(v)$.
Consider the following $n$ permutations of $v_1, \ldots, v_n$,
where the first $n-1$ permutations consist of $v_1$ followed by
cyclic permutations of $v_2, \ldots, v_n$. The last permutation
is obtained by swapping $v_{n-1}$ and $v_n$ in the first
permutation:
\begin{equation}
\begin{array}{l@{\hspace{1ex}}l@{\hspace{1ex}}l@{\hspace{1ex}}c@{\hspace{1ex}}l@{\hspace{1ex}}l@{\hspace{1ex}}l}
v_1, & v_2,     & v_3,     & \ldots, & v_{n-2}, & v_{n-1}, & v_n \\
v_1, & v_3,     & v_4,     & \ldots, & v_{n-1}, & v_n,     & v_2     \\
v_1, & v_4,     & v_5,     & \ldots, & v_n,     & v_2,     & v_3 \\
 & & & \vdots & & & \\
v_1, & v_{n-1}, & v_n,     & \ldots, & v_{n-4}, & v_{n-3}, & v_{n-2} \\
v_1, & v_n,     & v_2,     & \ldots, & v_{n-3}, & v_{n-2}, & v_{n-1} \\
v_1, & v_2,     & v_3,     & \ldots, & v_{n-2}, & v_n,     & v_{n-1}
\end{array} \label{eq:perm0}
\end{equation}
The rows of the following matrix correspond to circuit
representations of the above permutations.  Thus row $i$ contains
the values $x_1, \ldots, x_n$ for the $i$th permutation in
(\ref{eq:perm0}).
\begin{equation}
\left[
\begin{array}{ccccccc}v_2 & v_3 & v_4 & \cdots & v_{n-1} & v_n     & v_1 \\
                      v_3 & v_1 & v_4 & \cdots & v_{n-1} & v_n     & v_2 \\
                      v_4 & v_3 & v_1 & \cdots & v_{n-1} & v_n     & v_2 \\
                   \vdots & \vdots
                                & \vdots
                                      &        & \vdots  & \vdots  & \vdots \\
                  v_{n-1} & v_3 & v_4 & \cdots & v_1     & v_n     & v_2 \\
                  v_n     & v_3 & v_4 & \cdots & v_{n-1} & v_1     & v_2 \\
                  v_2     & v_3 & v_4 & \cdots & v_n     & v_1     & v_{n-1}
\end{array}
\right]  \label{eq:00}
\end{equation}
Since each row of (\ref{eq:00}) is a point in $H_n(v)$, it suffices
to show that the rows are affinely independent.  Subtract $[v_n
\;\; v_3 \;\; v_4 \; \cdots \; v_{n-1} \;\; v_n \;\; v_2]$ from
every row of (\ref{eq:00}) to obtain
\begin{equation}
{\small \left[
\begin{array}{c@{\hspace{1ex}}c@{\hspace{1ex}}ccccc}
v_2-v_n     & 0       & 0       & \cdots & 0               & 0           & v_1-v_2 \\
v_3-v_n     & v_1-v_3 & 0       & \cdots & 0               & 0           & 0 \\
v_4-v_n     & 0       & v_1-v_4 & \cdots & 0               & 0           & 0 \\
\vdots      & \vdots  & \vdots  &        & \vdots          & \vdots      & \vdots \\
v_{n-1}-v_n & 0       & 0       & \cdots & v_1-v_{n-1}     & 0           & 0 \\
0           & 0       & 0       & \cdots & 0               & v_1-v_n     & 0 \\
v_2-v_n     & 0       & 0       & \cdots & v_n-v_{n-1}     & v_1-v_n     & v_{n-1}-v_2
\end{array}
\right] } \label{eq:01}
\end{equation}
The rows of (\ref{eq:00}) are affinely independent if and only if
the rows of (\ref{eq:01}) are.  It now suffices to show that
(\ref{eq:01}) is nonsingular, and we do so through a series of row
operations.  The first step is to subtract $(v_{n-1}-v_2)/(v_1-v_2)$
times row 1, $(v_n-v_{n-1})/(v_1-v_{n-1})$ times row $n-2$, and
row $n-1$ from row $n$ to obtain
\begin{equation}
{\small \left[
\begin{array}{ccccccc}
v_2-v_n     & 0       & 0       & \cdots & 0           & 0           & v_1-v_2 \\
v_3-v_n     & v_1-v_3 & 0       & \cdots & 0           & 0           & 0 \\
v_4-v_n     & 0       & v_1-v_4 & \cdots & 0           & 0           & 0 \\
\vdots      & \vdots  & \vdots  &        & \vdots      & \vdots      & \vdots  \\
v_{n-1}-v_n & 0       & 0       & \cdots & v_1-v_{n-1} & 0           & 0 \\
0           & 0       & 0       & \cdots & 0           & v_1-v_n     & 0 \\
E_n         & 0       & 0       & \cdots & 0           & 0 & 0
\end{array}
\right] } \label{eq:02}
\end{equation}
where
\[
E_n  = -\frac{v_n-v_{n-1}}{v_{n-1}-v_1}(v_n-v_{n-1})
-\frac{v_{n-1}-v_1}{v_2-v_1}(v_n-v_2)
\]
Interchange the first and last rows of (\ref{eq:02}) to obtain
\begin{equation}
{\small \left[
\begin{array}{ccccccc}
E_n         & 0       & 0       & \cdots & 0           & 0           & 0 \\
v_1-v_n     & v_1-v_3 & 0       & \cdots & 0           & 0           & 0 \\
v_4-v_n     & 0       & v_1-v_4 & \cdots & 0           & 0           & 0 \\
\vdots      & \vdots  & \vdots  &        & \vdots      & \vdots      & \vdots \\
v_{n-1}-v_n & 0       & 0       & \cdots & v_1-v_{n-1} & 0           & 0 \\
0           & 0       & 0       & \cdots & 0           & v_1-v_n     & 0 \\
v_2-v_n     & 0       & 0       & \cdots & 0           & 0           & v_1-v_2
\end{array}
\right] } \label{eq:03}
\end{equation}
Note that $E_n<0$ since $v_1<\cdots<v_n$.  Thus (\ref{eq:03}) is
a lower triangular matrix with nonzero diagonal elements and is
therefore nonsingular. 
\end{proof}

\section{Facet-Defining Inequalities}

We now develop necessary and sufficient conditions for an inequality containing at most $n-4$ variables to be facet defining for the hamiltonian circuit polytope.  The
following lemma is key.

\begin{lemma} \label{le:zeroes}
Suppose that the inequality
\begin{equation}
\sum_{j\in J} a_jx_j \geq \alpha \label{eqineq}
\end{equation}
is valid for circuit$(x_1, \ldots, x_n)$ and is satisfied as an
equation by at least one circuit $x$.  If $|J| \leq n-4$ and
\begin{equation}
\sum_{j=1}^n d_jx_j = \delta \label{eqeq}
\end{equation}
is satisfied by all circuits $x$ that satisfy (\ref{eqineq}) as an
equation, then $d_i=d_j$ for all $i,j\not\in J$.
\end{lemma}

\begin{proof}
Because $|J| \leq n-4$, it suffices to prove that
$d_{j_1}=d_{j_2}=d_{j_3}=d_{j_4}$ for any four distinct indices
$j_1, \ldots, j_4 \not\in J$.

Let $x^0$ be any circuit that satisfies (\ref{eqineq}) as an
equation, and let the permutation described by $x^0$ be
\vspace{-1ex}
\[
\begin{array}{l}
\hspace{-1.2ex} v_1,\ldots, v_{j_1-1},v_{j_1},v_{j_1+1},\ldots,v_{j_2-1},v_{j_2},
v_{j_2+1},\ldots,v_{j_3-1},v_{j_3},v_{j_3+1}, \ldots, v_{j_4-1},
v_{j_4}
\end{array}
\label{perm0}
\]
Consider the circuits $x^1, \ldots, x^5$ that describe the following
permutations, respectively:
\[
\begin{array}{l}
\hspace{-1.2ex} v_1,\ldots,v_{j_1-1},v_{j_1},v_{j_3+1},\ldots,v_{j_4-1},v_{j_4},v_{j_1+1},\ldots,v_{j_2-1},v_{j_2},v_{j_2+1},\ldots,v_{j_3-1},v_{j_3} \vspace{.5ex} \\
\hspace{-1.2ex} v_1,\ldots,v_{j_1-1}, v_{j_1},v_{j_2+1},\ldots,v_{j_3-1},v_{j_3},v_{j_3+1},\ldots,v_{j_4-1},v_{j_4},v_{j_1+1},\ldots,v_{j_2-1},v_{j_2} \vspace{.5ex} \\
\hspace{-1.2ex} v_1,\ldots, v_{j_1-1},v_{j_1},v_{j_2+1},\ldots,v_{j_3-1},v_{j_3},v_{j_1+1},\ldots,v_{j_2-1},v_{j_2},v_{j_3+1},\ldots,v_{j_4-1},v_{j_4} \vspace{.5ex} \\
\hspace{-1.2ex} v_1,\ldots, v_{j_1-1},v_{j_1},v_{j_1+1},\ldots,v_{j_2-1},v_{j_2},v_{j_3+1},\ldots,v_{j_4-1},v_{j_4}, v_{j_2+1},\ldots,v_{j_3-1},v_{j_3} \vspace{.5ex} \\
\hspace{-1.2ex} v_1,\ldots,v_{j_1-1},v_{j_1},v_{j_3+1},\ldots,v_{j_4-1},v_{j_4},v_{j_2+1},\ldots,v_{j_3-1},v_{j_3},v_{j_1+1},\ldots,v_{j_2-1},v_{j_2}
\vspace{.5ex}
\end{array}
\]
We obtain $x^1, \ldots, x^5$ from $x^0$ by viewing the permutation
represented by $x^0$ as a concatenation of four subsequences,
each ending in one of the values $v_{j_i}$.  We fix the first
subsequence and obtain $x^1$ and $x^2$ by cyclically permuting the
remaining three subsequences.  We obtain $x^3$, $x^4$ and $x^5$ by
interchanging a pair of subsequences.

Note that variables $x_{j_1}, \ldots, x_{j_4}$ have the values shown
below in each circuit $x^i$:
\[
\begin{array}{ccccl}
x_{j_1}     & x_{j_2}     & x_{j_3}     & x_{j_4} &\vspace{.5ex} \\
\cline{1-4}
v_{j_1+1}   & v_{j_2+1}   & v_{j_3+1}   & v_{1}       & (x^0) \\
v_{j_3+1}   & v_{j_2+1}   & v_{1}       & v_{j_1+1}   & (x^1) \\
v_{j_2+1}   & v_{1}       & v_{j_3+1}   & v_{j_1+1}   & (x^2) \\
v_{j_2+1}   & v_{j_3+1}   & v_{j_1+1}   & v_{1}       & (x^3) \\
v_{j_1+1}   & v_{j_3+1}   & v_{1}       & v_{j_2+1}   & (x^4) \\
v_{j_3+1}   & v_{1}       & v_{j_1+1}   & v_{j_2+1}   & (x^5)
\end{array}
\]
and all other variables $x_j$ have value $x^0_j$ in each circuit
$x^i$. Thus all six circuits $x^0, \ldots,x^5$ satisfy
(\ref{eqineq}) at equality, so that $dx^i=\delta$ for
$i=0,\ldots,5$.  This implies
\[
{\textstyle \frac{1}{2}}
 \left[
      \begin{array}{c}
      (dx^0+dx^1+dx^5)-(dx^2+dx^3+dx^4) \vspace{.5ex} \\
      (dx^0+dx^2+dx^5)-(dx^1+dx^3+dx^4) \vspace{.5ex} \\
      (dx^0+dx^3+dx^5)-(dx^1+dx^2+dx^4) \vspace{.5ex} 
      \end{array}
\right]
\begin{array}{c}
=
\left[
      \begin{array}{@{}c@{}}
      0 \\ 0 \\ 0 
      \end{array}
\right] \\
\ \vspace{-1.7ex}
\end{array}
\]
Substituting the values of $x^0, \ldots, x^5$, we obtain
\[
\begin{array}{@{}c@{}}
\left[
      \begin{array}{cccc}
      v_{j_3+1}-v_{j_2+1} & v_{j_2+1}-v_{j_3+1} & 0                    & 0 \\
      0                   & v_{1}-v_{j_3+1}     & v_{j_3+1}- v_{1}     & 0 \\
      0                   & 0                   & v_{j_1+1}- v_{1}     & v_{1}-v_{j_1+1}  
      \end{array}
\right] 
\\
\ \vspace{-.5ex} 
\end{array}
\left[
      \begin{array}{@{}c@{}}
      d_{j_1} \\ d_{j_2} \\ d_{j_3} \\ d_{j_4}
      \end{array}
\right]  
\begin{array}{@{}c@{}}
=
\left[
      \begin{array}{@{}c@{}}
      0 \\ 0 \\ 0 
      \end{array}
\right] \\
\ \vspace{-.5ex}
\end{array}
\]
from which we can conclude that  $d_{j_1}=d_{j_2}=d_{j_3}=d_{j_4}$.
$\Box$
\end{proof}

Lemma~\ref{le:zeroes} applies only when $|J| \leq n-4$ because its proof relies on the absence of at least four variables from (\ref{eqineq}).  The theorems  below are therefore stated only for $|J| \leq n-4$.  We conjecture that they also hold for the densest facets ($|J|>n-4$), but proof seems to require the analysis of several special cases that substantially complicate the argument.  This slightly stronger result would be of little additional value for identifying useful families of facets.

For a given $x$, we denote by $x(J)$ the tuple
$(x_{j_1}, \ldots, x_{j_m})$ when $J=\{j_1, \ldots, j_m\}$.  We say
that $x(J)$ is a {\em \mbox{$J$-circuit}} if it creates no cycles and is
therefore a partial solution of the circuit constraint.  That is,
$x(J)$ is a \mbox{$J$-circuit} if there is no subsequence $j_{i_1}, \ldots,
j_{i_k}$ of the indices in $J$ such that $x_{j_{i_t}}=v_{j_{i_{t+1}}}$
for $t=1, \ldots, k-1$ and $x_{j_{i_k}}=v_{j_{i_1}}$.  The following
lemma is straightforward, but its proof introduces notation we will
need later.

\begin{lemma} \label{le:projection}
If $\bar{x}(J)$ is a \mbox{$J$-circuit}, then there is a circuit $x$ such
that \mbox{$x(J)=\bar{x}(J)$}.
\end{lemma}

\begin{proof}
Let $J=\{j_1, \ldots, j_m\}$, and let $\{v_{i_1},
\ldots, v_{i_r}\}$ be the subset of domain values $v_1, \ldots,
v_n$ that occur in neither $\{v_{j_1}, \ldots, v_{j_m}\}$ nor
$\{ \bar{x}_{j_1},\ldots,\bar{x}_{j_m}\}$.  Consider the directed
graph $G_{\bar{x}(J)}$ that contains a vertex $v_i$ for each
$i\in\{1, \ldots, \mbox{$n$}\}$, a directed edge
$(v_{j_k},\bar{x}_{j_k})$ for $k=1, \ldots, m$, and a directed edge
$(v_{i_k},v_{i_{k+1}})$ for each $k=1, \ldots, r-1$.  The maximal
subchains of $G_{\bar{x}(J)}$ have the form
\[
\begin{array}{l}
v_{j_{k_1}} \rightarrow \cdots \rightarrow v_{j_{k'_1}} \rightarrow \bar{x}_{j_{k'_1}} \vspace{.5ex} \\
v_{j_{k_2}} \rightarrow \cdots \rightarrow v_{j_{k'_2}} \rightarrow \bar{x}_{j_{k'_2}} \\
\vdots \\
v_{j_{k_p}} \rightarrow \cdots \rightarrow v_{j_{k'_p}} \rightarrow \bar{x}_{j_{k'_p}} \vspace{.5ex} \\
v_{i_1} \rightarrow \cdots \rightarrow v_{i_r}
\end{array}
\]
Because
maximal subchains are disjoint, we can form a hamiltonian circuit in
$G_{\bar{x}(J)}$ by linking the last element of each subchain to the
first element of the next, and linking $v_{i_r}$ to $v_{k_1}$.  Let
$v_{s_1}, \ldots, v_{s_n}$ be the resulting circuit.  Then if
$x$ is given by $x_i = v_{s_{((i-1)\,\mbox{\tiny mod} \, n) + 1}}$ for
$i=1, \ldots, n$, then $x$ is a circuit and $x(J)=\bar{x}(J)$.
$\Box$
\end{proof}

The concept of {\em domination} between $J$-circuits is central to identifying facets of $H_n(v)$, because inequality (\ref{eqineq}) is valid if and only if it is satisfied by all undominated \mbox{$J$-circuits}.  If $(J_+,J_-)$ is a partition of $J$, we say that $x(J)$ dominates $y(J)$ with respect to $(J_+,J_-)$ when $x_j\leq y_j$ for all $j\in J_+$ and $x_j\geq y_j$ for all $j\in J_-$.  A
\mbox{$J$-circuit} $x(J)$ is {\em undominated} with respect to $(J_+,J_-)$ if no other \mbox{$J$-circuit} dominates it with respect to $(J_+,J_-)$.

\begin{lemma} \label{le:valid}
Inequality (\ref{eqineq}) is valid for the hamiltonian circuit polytope if and only if it is satisfied by all undominated \mbox{$J$-circuits} with respect to $(J_+,J_-)$, where $J_+=\{j\;|\;a_j>0\}$ and $J_-=\{j\;|\;a_j<0\}$.  
\end{lemma}

\begin{proof}
A valid inequality must be satisfied by all circuits.  This means, due to Lemma~\ref{le:projection}, that it must be satisfied by all \mbox{$J$-circuits} and therefore by all undominated \mbox{$J$-circuits}.  For the converse, suppose (\ref{eqineq}) is satisfied by all undominated \mbox{$J$-circuits}, and let $x$ be any circuit.  Then $x(J)$ is dominated by some undominated \mbox{$J$-circuit} $x'(J)$ with respect to $(J_+,J_-)$, which means that $a_j(x_j - x'_j) \geq 0$ for all $j\in J$.
Thus we have
\[
\sum_{j \in J} a_jx_{j} \geq \sum_{j \in J} a_j x'_{j} \geq \alpha
\]
because $x'(J)$ satisfies (\ref{eqineq}), and so $x$ satisfies
(\ref{eqineq}).  This shows (\ref{eqineq}) is valid. 
$\Box$
\end{proof}

The following theorem provides sufficient conditions under which an inequality is facet defining.

\begin{theorem}  \label{th:main}
Consider any inequality of the form (\ref{eqineq}).  Let $S$ be the set of \mbox{$J$-circuits} that are undominated with respect
to $(J_+,J_-)$, where $J_+=\{j\;|\;a_j>0\}$, $J_-=\{j\;|\;a_j<0\}$, and $1\leq |J|\leq n-4$.  If all \mbox{$J$-circuits} in $S$ satisfy (\ref{eqineq}) and at least $|J|$ affinely independent \mbox{$J$-circuits} satisfy
\begin{equation}
\sum_{j \in J} a_j x_{j} = \alpha
\label{+facet_eq}
\end{equation}
then (\ref{eqineq}) defines a facet of $H_n(v)$.
\end{theorem}\label{th:undom_facets}

\begin{proof}
Inequality (\ref{eqineq}) is valid by Lemma~\ref{le:valid}.  To show (\ref{eqineq}) is facet defining, let (\ref{eqeq}) be any equation satisfied by all circuits $x$
that satisfy (\ref{eqineq}) at equality.  Recall that all circuits satisfy (\ref{eq:affine}).  It suffices to show that (\ref{eqeq}) is a linear combination of (\ref{+facet_eq}) and (\ref{eq:affine}).

Let $S=\{x^1(J), \ldots, x^m(J)\}$.  Because $|J|\geq 1$ and $S$ is therefore nonempty, at least one \mbox{$J$-circuit} $x^i(J)\in S$ satisfies (\ref{eqineq}) at equality.  Lemma~\ref{le:projection} therefore implies that at least one circuit $x^i$ satisfies (\ref{eqineq})
at equality.  Thus since $|J| \leq n-4$, we have from
Lemma~\ref{le:zeroes} that $d_i=d_j$ for all $i,j \notin J$.

We first suppose that $d_j=0$ for all $j\notin J$.  Then (\ref{eqeq}) has the form
\begin{equation}
\sum_{j \in J} d_j x_j=\delta 
\label{eqeq1}
\end{equation}
Because $|J|$ affinely independent \mbox{$J$-circuits} satisfy (\ref{+facet_eq}) and therefore (\ref{eq_dJ}), these two equations are
the same up to a scalar multiple.  Thus (\ref{eqeq}) is a linear combination of (\ref{+facet_eq}) and (\ref{eq:affine}), where the latter has multiplier zero.

We now suppose that $d_j\neq 0$ for $j\notin J$.  Because the $d_j$s are equal for all $j\notin J$, we can without loss of generality write (\ref{eqeq}) as
\[
\sum_{j\in J} d_jx_j + \sum_{j\notin J} x_j = \delta
\]
This is a linear combination of (\ref{+facet_eq}) and (\ref{eq:affine}) if the following is a scalar multiple of (\ref{+facet_eq}):
\begin{equation}
\sum_{j\in J} (d_j-1)x_j = \delta - \sum_{j=1}^n v_j
\label{eq_dJ2}
\end{equation}
But this follows from the fact that $|J|$ affinely independent \mbox{$J$-circuits} satisfy (\ref{+facet_eq}) and (\ref{eq_dJ2}).
$\Box$
\end{proof}

A simple corollary sometimes suffices to show that inequalities are facet defining.

\begin{corollary} \label{co:main}
If $J$ is as in Theorem~\ref{th:main}, (\ref{eqineq}) is valid, and at least $|J|$ affinely independent \mbox{$J$-circuits} satisfy (\ref{eqineq}) at equality, then (\ref{eqineq}) is facet defining.
\end{corollary}

\begin{proof}
If (\ref{eqineq}) is valid, then it is satisfied by all undominated \mbox{$J$-circuits}, and the conditions of Theorem~\ref{th:main} apply.  
$\Box$
\end{proof}

To apply Theorem~\ref{th:main} (or Corollary~\ref{co:main}), one must identify a set of affinely independent $J$-circuits.  However, the number of circuits required is only the number $|J|$ of terms included in the facet-defining inequality, as opposed to $n$ circuits in traditional arguments based on affine independence.  The theorem can therefore be regarded as a lifting lemma.  It will allow us to exploit patterns in the selection of terms to be included, so as to establish several classes of facets.

Finally, we note that the conditions of Theorem~\ref{th:main} are necessary as well as sufficient for (\ref{eqineq}) to be facet defining.

\begin{theorem} \label{th:main2}
Consider any inequality (\ref{eqineq}) that is facet-defining for a
hamiltonian circuit polytope $H_n(v)$.  Let $J_+=\{j\;|\;a_j>0\}$ and
$J_-=\{j\;|\;a_j<0\}$.  Then (\ref{eqineq}) is satisfied by all undominated \mbox{$J$-circuits} with respect to $(J_+,J_-)$, and at least $|J|$ affinely independent \mbox{$J$-circuits} satisfy (\ref{+facet_eq}).
\end{theorem}

\begin{proof}
Because (\ref{eqineq}) is valid, Lemma~\ref{le:valid} implies that it is satisfied by all undominated \mbox{$J$-circuits}.  Furthermore, because (\ref{eqineq}) is facet defining, it is 
satisfied at equality by $n$ affinely independent circuits
$\bar{x}^1, \ldots, \bar{x}^n$.  Then $\{\bar{x}^1(J), \ldots,
\bar{x}^n(J)\}$ contains some subset $\{\bar{x}^{j_1}(J), \ldots,
\bar{x}^{j_m}(J)\}$ of $|J|=m$ affinely independent
\mbox{$J$-circuits}, which satisfy (\ref{eqineq}).  
$\Box$
\end{proof}

\section{Generating Undominated Circuits}
\label{greedy}

A simple greedy procedure can be used to generate all \mbox{$J$-circuits}
$\bar{x}(J)$ that are undominated with respect to $(J_+,J_-)$.
It is applied for each ordering $j_1, \ldots, j_m$ of the elements
of $J$.  First, let $\bar{x}_{j_1}$ be the smallest domain value
$v_i$ if $j_1\in J_+$, or the largest if $j_1\in J_-$.  Then let
$\bar{x}_{j_2}$ be the smallest (or largest) remaining domain value
that does not create a cycle.  Continue until all $\bar{x}_j$ for
$j\in J$ are defined.  The precise algorithm appears in
Fig.~\ref{fig:greedy}.

\begin{figure}
\centering
\fbox{
\parbox[c]{6in}{
\begin{tabbing}
xxx \= xxx \= xxx \= xxx \= \kill
For each ordering $j_1, \ldots, j_m$ of the elements of $J$: \\
\> Let $\bar{J}=\{1, \ldots, n\}$ and $J'=\emptyset$. \\
\> For $i=1, \ldots, m$: \\
\> \> Add $j_i$ to $J'$. \\
\> \> If $j_i \in J_+$ then let $\bar{x}_{j_i}$ be the minimum value $v_k$ in $\{v_i\;|\;i\in \bar{J}\}$ \\
\> \> \> such that $\bar{x}(J')$ is a $J'$-circuit. \\
\> \> Else let $\bar{x}_{j_i}$ be the maximum value $v_k$ in $\{v_i\;|\;i\in \bar{J}\}$ \\
\> \> \> such that $\bar{x}(J')$ is a $J'$-circuit. \\
\> \> Remove $k$ from $\bar{J}$. \\
\> Add $\bar{x}(J)$ to the list of undominated \mbox{$J$-circuits}.
\end{tabbing}
}
}
\vspace{-1ex}
\caption{Greedy procedure for generating undominated \mbox{$J$-circuits}.  Input: tuple $v$ of domain values, index set $J$, and partition $(J_+,J_-)$ of $J$.  Output: a complete list of \mbox{$J$-circuits} that are undominated with respect to $(J_+,J_-)$.}
\label{fig:greedy}
\vspace{2ex}
\end{figure}

To prove that the greedy procedure is correct, it is convenient to write $x_j\prec y_j$ when either $x_j < y_j$ and $j\in J_+$ or $x_j > y_j$ and $j\in J_-$.

\begin{theorem} \label{th:greedy}
The greedy procedure of Fig.~\ref{fig:greedy} generates \mbox{$J$-circuits}
that are undominated with respect to $(J_+,J_-)$.
\end{theorem}

\begin{proof}
Let $\bar{x}(J)$ be a \mbox{$J$-circuit} generated by the
procedure for a given ordering $j_1, \ldots, j_m$.  To see that
$\bar{x}(J)$ is undominated with respect to $(J_+,J_-)$, assume
otherwise. Then there exists a \mbox{$J$-circuit} $\bar{y}(J)$ that dominates $\bar{x}(J)$ such that $\bar{y}_{j_t} \prec \bar{x}_{j_t}$ for some $t\in\{1, \ldots, m\}$. Let $t$ be the
smallest such index, so that $\bar{x}_{j_k} = \bar{y}_{j_k}$ for
$k=1, \ldots, t-1$.  This contradicts the greedy construction of
$\bar{x}$, because $\bar{y}_{j_t}$ is available when $\bar{x}_{j_t}$
is assigned to $x_{j_t}$.  $\Box$
\end{proof}

As an example, consider circuit$(x_1,\ldots,x_7)$ where each $x_j$ has domain $\{v_1, \ldots, v_7\}$.  The undominated \mbox{$J$-circuits} of $J=\{1,3,4\}$ with respect to $(J,\emptyset)$ can be generated by considering the six orderings of $1,3,4$ listed on the left below.  The
resulting undominated \mbox{$J$-circuits} appear on the right.
\[
\begin{array}{c@{\hspace{5ex}}c}
(j_1,j_2,j_3) & (x_1,x_3,x_4) \\
\ \vspace{-2.4ex} \\
\hline \vspace{-2.4ex} \\
(1,3,4) & (v_2,v_1,v_3) \vspace{.5ex} \\
(1,4,3) & (v_2,v_4,v_1) \vspace{.5ex} \\
(3,1,4) & (v_2,v_1,v_3) \vspace{.5ex} \\
(3,4,1) & (v_4,v_1,v_2) \vspace{.5ex} \\
(4,1,3) & (v_2,v_4,v_1) \vspace{.5ex} \\
(4,3,1) & (v_3,v_2,v_1) 
\end{array}
\]
There is only one undominated $J$-circuit with respect to $(\{1,3\},\{4\})$, \mbox{because} all six orderings result in the
same \mbox{$J$-circuit} $(v_2,v_1,v_7)$.

It remains to show that the greedy procedure finds all undominated \mbox{$J$-circuits}. We will first prove this for the partition $(J,\emptyset)$ because the \mbox{argument} simplifies considerably in this case.  Thus we assume that circuit $x$ dominates circuit $x'$ when $x\leq x'$.  The proof for the general case appears in the Appendix.

\begin{theorem}  \label{th:simplifiedgreedycomplete} Any undominated $J$-circuit with respect to $(J,\emptyset)$ can be generated in a greedy fashion for some ordering of the indices in $J$.
\end{theorem}

\begin{proof}
Let $\bar{x}(J)$ be a $J$-circuit that is undominated with
respect to $(J,\emptyset)$.  Let $J=\{i_1, \ldots,
i_m\}$ where $\bar{x}_{i_1} < \cdots < \bar{x}_{i_m}$, and let $y=(y_{i_1}, \ldots, y_{i_m})$ be the greedy solution with respect to the ordering $i_1, \ldots, i_m$. We claim that $\bar{x}_{i_{\ell}}=y_{i_{\ell}}$ for $\ell=1,
\ldots, m$, which suffices to prove the theorem.  

Supposing to the contrary, let $t$ be the smallest index for which \mbox{$\bar{x}_{i_t}\neq y_{i_t}$}.  Clearly $\bar{x}_{i_t} < y_{i_t}$ is inconsistent with the greedy choice, because $\bar{x}_{i_t}$ is available when $y_{i_t}$ is assigned a value. 
Thus we have $\bar{x}_{i_t} > y_{i_t}$.
By hypothesis, $\bar{x}$ is undominated with respect to $J$.  
We therefore have $\bar{x}_{i_{\ell}} < y_{i_{\ell}}$ for some $\ell\in\{t+1, \ldots, m\}$.
Let $u$ be the smallest such index.   
Finally, let $t'$ be the largest index in $\{t, \ldots, u-1\}$ such that $\bar{x}_{i_{t'}} > y_{i_{t'}}$.  
We know that $t'$ exists because $\bar{x}_{i_t} > y_{i_t}$.  
Thus we have two sequences of values related as follows: 
\[
\begin{array}{c@{\;}c@{\;}c@{\;}c@{\;}c@{\;}c@{\;}c@{\;}c@{\;}c@{\;}c@{\;}c@{\;}c@{\;}c@{\;}c@{\;}c@{\;}c@{\;}c@{\;}c@{\;}c}
\bar{x}_{i_1} & < & \cdots & < & 
\bar{x}_{i_{t-1}} & < & \bar{x}_{i_t} & < 
& \cdots & < & \bar{x}_{i_{t'-1}} & < & 
\bar{x}_{i_{t'}} & < & \cdots & < & 
\bar{x}_{i_{u-1}} & < & \bar{x}_{i_u} \\
$\rotatebox[origin=c]{270}{$=$}$ &   &        &   & 
$\rotatebox[origin=c]{270}{$=$}$ &   & $\rotatebox[origin=c]{270}{$>$}$ &   & 
  &   & $\rotatebox[origin=c]{270}{$\geq$}$ &   &  
$\rotatebox[origin=c]{270}{$>$}$ &   &    &   & 
$\rotatebox[origin=c]{270}{$\geq$}$ &      & $\rotatebox[origin=c]{270}{$<$}$  \\
y_{i_1} &   & \cdots &   & y_{i_{t-1}} &   &
y_{i_t} &   & \cdots &   & y_{i_{t'-1}} &   &
y_{i_{t'}} &   & \cdots &   & y_{i_{u-1}} &    &
y_{i_u}
\end{array}
\]

We first show that value $\bar{x}_{i_u}$ has not yet been assigned
in the greedy algorithm when $y_{i_u}$ is assigned a value.  That is, we show that $\bar{x}_{i_u}\not\in\{y_{i_1}, \ldots, y_{i_{u-1}}\}$.
Suppose to the contrary that $\bar{x}_{i_u}=y_{i_{w}}$ for some $w\in \{1 \ldots, u-1\}$.
But this is impossible, because $\bar{x}_{i_u}>\bar{x}_{i_w}\geq y_{i_w}$.  
We next show that value $\bar{x}_{i_{t'}}$ has not yet been assigned
in the greedy algorithm when $y_{i_u}$ is assigned a value.  That is, we show that $\bar{x}_{i_{t'}}\not\in \{y_{i_1}, \ldots, y_{i_{u-1}}\}$.  
To begin with, we have that $\bar{x}_{i_{t'}}\not\in \{y_{i_1}, \ldots,
y_{i_{t'-1}}\}$, by virtue of the same reasoning just applied
to $\bar{x}_{i_u}$.  Also $\bar{x}_{i_{t'}}\neq y_{i_{t'}}$,
since by hypothesis $\bar{x}_{i_{t'}} > y_{i_{t'}}$.  To show
that $\bar{x}_{i_{t'}}\not\in \{y_{i_{t'+1}}, \ldots,
y_{i_{u-1}}\}$, suppose to the contrary that $\bar{x}_{i_{t'}} = y_{i_w}$ for some $w\in \{t'+1, \ldots, u-1\}$.  
Then since $\bar{x}_{i_{t'}} < \bar{x}_{i_w}$, we must have $\bar{x}_{i_w} >
y_{i_w}$.  But this contradicts the definition of $t'$ ($< w$)
as the largest index in $\{1, \ldots, u-1\}$ such that
$\bar{x}_{i_{t'}} > y_{i_{t'}}$.  Thus $\bar{x}_{i_{t'}} \neq y_{i_w}$.  

Because $\bar{x}_{i_u} < y_{i_u}$ and value $\bar{x}_{i_u}$ has
not yet been assigned, setting $y_{i_u} = \bar{x}_{i_u}$ must create
a cycle in $y$, because otherwise setting $y_{i_u} = \bar{x}_{i_u}$
would have been the greedy choice.  Also, setting
$y_{i_u}=\bar{x}_{i_{t'}}$ was not the greedy choice because
$y_{i_u} > \bar{x}_{i_u} > \bar{x}_{i_{t'}}$.  Thus setting
$y_{i_u}=\bar{x}_{i_{t'}}$ must likewise create a cycle in
$y$, because $\bar{x}_{i_{t'}}$ has not yet been assigned. Now
define $G_{y(J)}$ as before and consider the maximal subchain
in $G_{y(J)}$ that contains $y_{i_u}$.  Let the segment
of the subchain up to $y_{i_u}$ be
\[
v_z \rightarrow \cdots \rightarrow v_{i_u} \rightarrow y_{i_u}
\]
Because setting $y_{i_u}=\bar{x}_{i_u}$ creates a cycle in
$y$, we must have $\bar{x}_{i_u} = v_z$.  Similarly,
because setting $y_{i_u}=\bar{x}_{i_{t'}}$ creates a cycle in
$y$, we must have $\bar{x}_{i_{t'}} = v_z$.  This implies
$\bar{x}_{i_u}=\bar{x}_{i_{t'}}$, which is impossible because
$\bar{x}_{i_u}>\bar{x}_{i_{t'}}$.  $\Box$
\end{proof}

\begin{theorem}  \label{th:greedycomplete} Any undominated $J$-circuit with respect to $(J_+,J_-)$ can be generated in a greedy fashion for some ordering of the indices in $J$.
\end{theorem}

\begin{proof}
See the Appendix. 
\end{proof}

\section{Permutation and Two-term Facets}
\label{permutation}

We begin by identifying two special classes of facets of
$H_n(v)$, namely, permutation facets and two-term facets.

The {\em permutohedron} $P_n(v)$ for an arbitrary domain $\{v_1,\ldots, v_n\}$ can be defined as the convex hull of all points whose coordinates are permutations of $v_1, \ldots, v_n$.  We refer to the facets of $P_n(v)$ as {\em permutation facets}.  The circuit
polytope $H_n(v)$ is contained in $P_n(v)$ because every circuit
$(x_1, \ldots, x_n)$ is a permutation of $v_1, \ldots, v_n$.
This means that every facet-defining inequality for $P_n(v)$ is
valid for circuit but not necessarily facet defining.  This raises
the question as to which permutation facets are also circuit facets.
We will identify a large family of permutation facets that can be
immediately recognized as circuit facets.

The permutohedron $P_n(v)$ has dimension $n-1$, and its affine hull is described by 
\begin{equation}
\sum_{j=1}^n x_j = \sum_{j=1}^n v_j
\label{eq:2termaffine}
\end{equation}
The facets of $P_n(v)$ are identified in \cite{Hoo00,WilYan01}, and they are
defined by
\begin{equation}
\sum_{j\in J} x_j \geq \sum_{j=1}^{|J|} \hspace{0ex} v_j
\label{eq:permfacet}
\end{equation}
for all $J\subset \{1, \ldots, n\}$ with $1\leq |J|\leq n-1$.
(Recall that $0\leq v_1 <\cdots < v_n$.)  This result is generalized
in \cite{Hooker12} to domains with more than $n$ elements.

For example, the permutohedron $P_3(v)$ with $v=(2,4,5)$ is
defined by
\[
\begin{array}{ll}
x_1+x_2+x_3=11 \vspace{.5ex} \\
x_i \geq 2, \;\mbox{for $i=1,2,3$} \vspace{.5ex} \\
x_i+x_j \geq 6, \; \mbox{for distinct $i,j\in\{1,2,3\}$}
\end{array}
\]
We can see at this point that a facet-defining inequality for
$P_n(v)$ need not be facet-defining for $H_n(v)$.  The inequality
$x_1+x_2\geq 6$ is facet-defining for $P_3(v)$ but not for $H_3(v)$,
which is the line segment from $(4,5,2)$ to $(5,2,4)$.  However, a large family of inequalities are facet defining for both $H_n(v)$ and $P_n(v)$.

\begin{theorem}\label{th:permfacet1}
The inequality (\ref{eq:permfacet}) defines a facet of $H_n(v)$ if
$1\leq |J| \leq n-4$ and $j>2$ for all $j\in J$.
\end{theorem}

\begin{proof}
Let $J=\{j_1, \ldots, j_m\}$.  Inequality (\ref{eq:permfacet}) is clearly valid because the variables $x_{j_1},\ldots,x_{j_m}$ must have pairwise distinct values.  By Corollary~\ref{co:main}, it suffices to exhibit $m$ affinely independent \mbox{$J$-circuits} that satisfy (\ref{eq:permfacet}) at equality.  Consider the following assignments to $(x_{j_1},\ldots,x_{j_m})$:

\begin{equation}
\begin{array}{ccccccc}
x_{j_1} & x_{j_2} & x_{j_3} & \cdots & x_{j_{m-1}} &  x_{j_m} \vspace{.5ex} \\
\hline
v_1     & v_2     & v_3     & \cdots & v_{m-1}     & v_m  \\
v_2     & v_1     & v_3     & \cdots & v_{m-1}     & v_m \\
v_1     & v_3     & v_2     & \cdots & v_{m-1}     & v_m \\
\vdots  & \vdots  & \vdots  &        & \vdots      & \vdots \\
v_1     & v_2     & v_3     & \cdots & v_m         & v_{m-1}
\end{array} \label{eq:fam31}
\end{equation}
The $i$th assignment is obtained from the first by swapping $v_{i-1}$ and $v_i$.
These \mbox{assign}ments obviously satisfy (\ref{eq:permfacet}) at equality.  They are also affinely independent,
as can be seen by subtracting the first row from each row.  It \mbox{remains} to show that the assignments create no cycles and are therefore \mbox{$J$-circuits}.  For this, it suffices to show that each $x_{j_i}$ is assigned a value $v_k$ with $k<j_i$.  The first assignment satisfies this condition because $2<j_1$ and $j_1<\cdots <j_m$ imply that $i<j_i-1$ for $i=1, \ldots, m$.  The $i$th \mbox{assign}ment agrees with the first on the \mbox{values} of all variables except $x_{j_{i-1}}, x_{j_i}$.  It sets  $(x_{j_{i-1}},x_{j_i})=(v_i,v_{i-1})$, which satisfies $i<j_{i-1}$ because $i-1<j_{i-1}-1$, and satisfies $i-1<j_i$ because $i<j_i-1$.  The $i$th assignment therefore satisfies the condition and is a \mbox{$J$-circuit} for $i=2, \ldots, m$. $\Box$
\end{proof}

Another special class of facet-defining inequalities are those
containing two terms, which can be listed in closed form.

\begin{corollary} \label{co:twoterm}
If $n\geq 6$, the two-term facets of $H_n(v)$ are precisely those
defined by
\begin{eqnarray}
&& x_i+x_j \geq v_1+v_2, \;\;\mbox{for distinct $i,j\in \{3,\ldots,n\}$} \vspace{.5ex} \label{eq:2term1} \\
&& (v_3-v_1)x_1 + (v_3-v_2)x_2 \geq v_3^2 - v_1v_2 \vspace{.5ex} \label{eq:2term2} \\
&& (v_2-v_1)x_2 + (v_3-v_1)x_i \geq v_2v_3 - v_1^2, \;\; \mbox{for $i\in \{3, \ldots, n\}$} \vspace{.5ex} \label{eq:2term3} \\
&& (v_{n-1}-v_{n-2})x_{n-1} + (v_n-v_{n-2})x_n \leq v_n v_{n-1} - v_{n-2}^2  \vspace{.5ex} \label{eq:2term5} \\
&& (v_n-v_{n-2})x_i + (v_n-v_{n-1})x_{n-1} \leq v_n^2 - v_{n-1}v_{n-2}, \label{eq:2term6} \\
&& \hspace{40ex} \mbox{for $i\in \{1, \ldots, n-2\}$} \nonumber
\end{eqnarray}
\end{corollary}

\vspace{-2ex}
\begin{proof}
Consider an arbitrary two-term inequality $a_ix_i+a_jx_j\geq \alpha$.  If we suppose $a_i,a_j>0$, four cases can be distinguished.  {\em Case 1:} $i,j>2$.  The two permutations of $i,j$ generate the two undominated \mbox{$J$-circuits} $(v_1,v_2)$ and $(v_2,v_1)$, where $J=\{i,j\}$.  The only equation satisfied by these two affinely independent \mbox{$J$-circuits}, up to a positive scalar multiple, is $x_i+x_j=v_1+v_2$.  So by Theorems~\ref{th:main} and~\ref{th:main2}, all facet-defining inequalities for this case have the form (\ref{eq:2term1}).  {\em Case 2:} $(i,j)=(1,2)$.  The undominated \mbox{$J$-circuits} are $(v_2,v_3)$ and $(v_3,v_1)$, which satisfy only (\ref{eq:2term2}) at equality, up to a positive scalar multiple.  {\em Case~3.} $i=1$, $j>2$.  The two permutations of $1,j$ generate the same undominated \mbox{$J$-circuit} $(v_2,v_3)$.  Thus no two affinely independent \mbox{$J$-circuits} satisfy $a_1x_1+a_jx_j=\alpha$, and by Theorem~\ref{th:main2} there are no facet-defining inequalities in this case.  {\em Case 4.} $i=2$, $j>2$.  The undominated \mbox{$J$-circuits} are $(v_1,v_2)$ and $(v_3,v_1)$, which satisfy only (\ref{eq:2term3}) at equality.  

Now if we suppose $a_i,a_j<0$, similar reasoning yields the facets (\ref{eq:2term5})--(\ref{eq:2term6}) and
\[
x_i+x_j \leq v_{n-1}+v_n, \;\;\mbox{for distinct $i,j\in \{1,\ldots,n-2\}$} 
\]
which is redundant of (\ref{eq:2term1}) because it is the sum of (\ref{eq:2term1}) and the negation of (\ref{eq:2termaffine}).  Finally, if $a_i>0$ and $a_j<0$, we consider four cases:  $i>1$ and $j<n$; $i=1$ and $j<n$; $i>1$ and $j=n$; and $(i,j)=(1,n)$.  The two permutations of $i,j$ generate only one \mbox{$J$-circuit} in each case, respectively $(v_1,v_n)$, $(v_2,v_n)$, $(v_1,v_{n-1})$, and $(v_2,v_{n-1})$.  This means by Theorem~\ref{th:main2} that there are no additional facets.  The situation is similar when $a_i<0$ and $a_j>0$. $\Box$
\end{proof}

\section{A Hierarchy of Facets}

We now describe a hierarchy of facets of increasing complexity.  To simplify discussion, we suppose in this section that each variable has domain $\{v_1,\ldots,v_n\}=\{1, \ldots, n\}$, and we consider only facets defined by inequalities with nonnegative coefficients.  We therefore focus on $H_n(u)$, where $u=(1, \ldots, n)$.

The intuition behind the hierarchy is as follows.  On level~0 of the hierarchy, the number of variables in an inequality (\ref{eqineq}) is less than the smallest index in $J$.  The undominated \mbox{$J$-circuits} are simply the permutations of $1, \ldots, m$, because the greedy algorithm of Section~\ref{greedy} never encounters a \mbox{cycle}.  As a result, the only facets on level~0 are permutation facets.   In higher levels of the hierarchy, the index of the first variable is smaller than the number of variables in the facet, which increases the combinatorial complexity of undominated \mbox{$J$-circuits} and yields more complicated facets.  We will exhaustively identify facets for levels~0, 1, and~2, although one can in principle use similar methods to identify facets on higher levels.

Let level~$d$ of the hierarchy consist of inequalities of the form
\begin{equation}
\sum_{j=m-d+1}^m \hspace{-2.5ex} a_jx_j + \hspace{-.7ex} \sum_{i=d+1}^m  \hspace{-1ex} a_{j_i}x_{j_i} \geq \alpha
\label{eq:fam1}
\end{equation}
where each $a_j>0$, where $m<j_{d+1}< \cdots < j_m$, and where $\{x_{j_{d+1}},\ldots,x_{j_m}\}$ is any subset of $m-d$ variables in $\{x_{m+1}, \ldots, x_n\}$.  Thus (\ref{eq:fam1}) contains $m$ variables, and $m-d$ variables are absent before the first variable.  Note also that the first $d$ variables are consecutive.  We will identify one family of facet-defining inequalities on level 0, two families on level 1, and five families on level 2.

First, we have immediately from Theorem~\ref{th:permfacet1} that level 0 contains a class of permutation facets.
\begin{corollary}
The following level~0 inequalities are facet defining for $H_n(u)$:
\[
\sum_{i=1}^m x_{j_i} \geq {\textstyle\frac{1}{2}} m(m+1), \;\;\; m=2, \ldots, n
\]
for any set $\{x_{j_1}, \ldots, x_{j_m}\}$ of $m$ variables in $\{x_{m+1}, \ldots, x_n\}$.
\end{corollary}

For level 1 we have the following.

\begin{theorem}
The following level 1 inequalities are facet defining for $H_n(u)$:
\begin{equation}
x_m + \sum_{i=2}^m x_{j_i} \geq {\textstyle\frac{1}{2}}m(m+1), \;\;\;m=3, \ldots, \lceil n/2 \rceil
\label{eq:fam2a} 
\end{equation}
\begin{equation}
x_m + 2\sum_{i=2}^m x_{j_i} \geq m^2+1, \;\;\;m=2, \ldots, \lceil n/2 \rceil
\label{eq:fam2} 
\end{equation}
for any subset $\{x_{j_2}, \ldots, x_{j_m}\}$ of $m-1$ variables in $\{x_{m+1}, \ldots, x_n\}$, provided \mbox{$n-m\geq 4$}.
\end{theorem}

\proof{Proof.} Here $J=\{m,j_2,\ldots,j_m\}$.  Inequality (\ref{eq:fam2a}) is facet defining due to Theorem~\ref{th:permfacet1}.  To show that (\ref{eq:fam2}) is facet defining, it suffices to show that it is satisfied by all undominated \mbox{$J$-circuits} and is satisfied at equality by $m$ affinely independent \mbox{$J$-circuits}.  From Theorem~\ref{th:simplifiedgreedycomplete}, all undominated \mbox{$J$-circuits} correspond to permutations of the elements of $J$, or equivalently, permutations $x'$ of $(x_m,x_{j_2},\ldots,x_{j_m})$.  We distinguish two cases: permutations in which $x_m$ is last, resulting in {\em type~1} circuits, and permutations in which $x_m$ is not last, resulting in {\em type~2} circuits.  Type~1 \mbox{$J$-circuits} have the form $x'=(1,\ldots,m-1,m+1)$, because once the first $m-1$ variables in $x'$ are assigned $1, \ldots, m-1$, $x_m$ cannot be assigned the next value $m$ and must be assigned $m+1$.  For all such \mbox{$J$-circuits}, the left-hand side of (\ref{eq:fam2}) has value 
\[
(m+1) + 2\left(1 + 2 + \cdots + (m-1)\right)  = m^2 + 1
\]
which satisfies (\ref{eq:fam2}).  Type~2 \mbox{$J$-circuits} have the form $x''=(1,\ldots,m)$ where $x''$ is any permutation of $(x_m,x_{j_2},\ldots,x_{j_m})$ in which $x_m$ is not last.  Because $x_m$ has the smallest coefficient in (\ref{eq:fam2}), the LHS of (\ref{eq:fam2}) is minimized over type~2 \mbox{$J$-circuits} when $x_m$ occurs next to last in $x''$, in which case the LHS has value
\[
(m-1) + 2\left( 1 + 2 + \cdots + (m-2) + m \right) = m^2+1
\]
Thus (\ref{eq:fam2}) is again satisfied.

We now exhibit $m$ affinely independent \mbox{$J$-circuits} satisfying (\ref{eq:fam2}) at equality.  The first $m-1$ \mbox{$J$-circuits} below are type~1, and the last is type 2:
\begin{equation}
\begin{array}{c@{\hspace{6ex}}cccccc}
x_m     & x_{j_2} & x_{j_3} & x_{j_4} & \cdots & x_{j_{m-1}} &  x_{j_m} \vspace{.5ex} \\
\hline
m-1     & 1       & 2       & 3       & \cdots & m-2         & m  \\
m-1     & 2       & 1       & 3       & \cdots & m-2         & m \\
m-1     & 1       & 3       & 2       & \cdots & m-2         & m \\
\vdots  & \vdots  & \vdots  & \vdots  &        & \vdots      & \vdots \\
m-1     & 1       & 2       & 3       & \cdots & m           & m-2 \\
\ \\
m+1     & 1       & 2       & 3       & \cdots & m-2         & m-1 
\end{array} \label{eq:fam29}
\end{equation}
These satisfy (\ref{eq:fam2}) at equality, as noted above.  The   $(m-1)\times(m-1)$ submatrix in the upper right is obtained by swapping pairs of elements in the first row. After suitable row operations, (\ref{eq:fam29}) becomes
\begin{equation}
\begin{array}{c@{\hspace{6ex}}cccccc}
x_m     & x_{j_2} & x_{j_3} & x_{j_4} & \cdots & x_{j_{m-1}} &  x_{j_m} \vspace{.5ex} \\
\hline
(m-1)/s & 1       & 0       & 0       & \cdots & 0           & 0  \\
(m-1)/s & 0       & 1       & 0       & \cdots & 0           & 0 \\
(m-1)/s & 0       & 0       & 1       & \cdots & 0           & 0 \\
\vdots  & \vdots  & \vdots  & \vdots  &        & \vdots      & \vdots \\
(m-1)/s & 0       & 0       & 0       & \cdots & 0           & 1 \\
\ \\
m+1     & 1       & 2       & 3       & \cdots & m-2         & m-1 
\end{array} \label{eq:fam29a}
\end{equation}
where $s=\frac{1}{2}m(m-1)+1$ is the sum of the entries in each row of the submatrix.  After further row operations, the last row is reduced to \mbox{$2+(m-1)/s$} followed by $m-1$ zeros, resulting in a triangular matrix (after \mbox{rearranging} columns) with nonzeros on the diagonal.  The matrix (\ref{eq:fam29}) is therefore nonsingular, and the rows are affinely independent. $\Box$
\endproof \medskip

Finally, we identify five classes of level~2 facets.  

\begin{theorem}
The following level~2 inequalities are facet defining for $H_n(u)$:
\begin{eqnarray}
&& \hspace{-8ex} x_{m-1} + x_m + \sum_{i=3}^m x_{j_i} \geq {\textstyle\frac{1}{2}}m(m+1), \;\;\; m=4, \ldots, \lceil (n+1)/2 \rceil \label{eq:fam10} \vspace{.5ex} \\
&& \hspace{-8ex} 2x_{m-1} + x_m + 2\sum_{i=3}^m x_{j_i} \geq m^2+1, \;\;\; m=4, \ldots, \lceil (n+1)/2 \rceil \label{eq:fam11} \vspace{.5ex}  \\
&& \hspace{-8ex} 2x_{m-1} + x_m + 4\sum_{i=3}^m x_{j_i} \geq m(2m-3)+5, \;\;\; m=3, \ldots, \lceil (n+1)/2 \rceil \label{eq:fam12} \vspace{.5ex} \\
&& \hspace{-8ex} 3x_{m-1} + 2x_m + 4\sum_{i=3}^m x_{j_i} \geq m(2m-1)+4, \;\;\; m=3, \ldots, \lceil (n+1)/2 \rceil \label{eq:fam13} \vspace{.5ex} \\
&& \hspace{-8ex} 3x_{m-1} + 2x_m + 5\sum_{i=3}^m x_{j_i} \geq {\textstyle\frac{5}{2}}m(m-1) + 6, \;\;\; m=3, \ldots, \lceil (n+1)/2 \rceil \label{eq:fam14}
\end{eqnarray}
for any given set $\{x_{j_3}, \ldots, x_{j_m}\}$ of $m-2$ variables in $\{x_{m+1}, \ldots, x_n\}$, if \mbox{$n-m\geq 4$}.
\end{theorem}

\proof{Proof.} Inequality (\ref{eq:fam10}) is facet defining due to Theorem~\ref{th:permfacet1}.  For the remaining inequalities we apply Theorem~\ref{th:main}.  First, we show that (\ref{eq:fam11})--(\ref{eq:fam14}) are satisfied by all undominated \mbox{$J$-circuits}, where $J=\{m-1,m,j_3, \ldots, j_m\}$.  This can be shown individually for each inequality, but we can establish the result for all at once by showing that 
\begin{equation}
ax_{m-1} + bx_m + c\sum_{i=3}^m x_{j_i} \geq \beta 
\label{eq:fam20}
\end{equation}
is satisfied by all undominated \mbox{$J$-circuits}, given that
\[
\vspace{-2ex}
\beta = (m-2)a + (m+1)b + {\textstyle\frac{1}{2}}(m-3)(m-2)c + (m-1)c 
\]
and 
\begin{equation}
2a\geq c, \;\;\; c\geq a, \;\;\; 3a\geq 2b+c, \;\;\; 2a\geq 3b, \;\;\; c\geq 2b
\label{eq:fam21}
\end{equation}
Note that the inequalities (\ref{eq:fam11})--(\ref{eq:fam14}) have the form (\ref{eq:fam20}) and satisfy the relations (\ref{eq:fam21}).  It can also be checked that $\beta$ is equal to the right-hand side of each inequality (\ref{eq:fam11})--(\ref{eq:fam14}).  It therefore suffices to show that all undominated \mbox{$J$-circuits} satisfy (\ref{eq:fam20}).

To show this, we again apply Theorem~\ref{th:simplifiedgreedycomplete}.   We partition permutations $x'$ of $x=(x_{m-1},x_m,x_{j_3}, \ldots, x_{j_m})$ into 5 classes, which give rise to 5 types of \mbox{$J$-circuits}.  It suffices to show that \mbox{$J$-circuits} of all 5 types satisfy (\ref{eq:fam20}).
\begin{description}
\item {\em Type 1.}  $x_m$ occurs last and $x_{m-1}$ next to last in $x'$.  Circuits constructed in a greedy fashion have the form $x'=(1,\ldots,m-2,m,m+1)$.  This is because once the first $m-2$ variables in $x'$ are assigned $1, \ldots, m-2$, variable $x_{m-1}$ cannot be assigned $m-1$ and is therefore assigned $m$.  Now $x_m$ cannot be assigned $m-1$ without creating a cycle with $x_{m-1}$ and is therefore assigned $m+1$.  The LHS of (\ref{eq:fam20}) is
\[
ma + (m+1)b + (m-2)c + {\textstyle\frac{1}{2}}(m-3)(m-2)c \geq \beta
\]
where the inequality follows from the fact that $2a\geq c$.  So \mbox{$J$-circuits} of type~1 satisfy (\ref{eq:fam20}).

\item {\em Type 2.}  $x_m$ occurs last but $x_{m-1}$ does not occur next to last in $x'$.  The circuits have the form $x'=(1,\ldots,m-1,m+1)$.  Because $a\leq c$, the LHS of (\ref{eq:fam20}) is minimized when $x_{m-1}$ occurs second from last in $x'$ (i.e., in position $m-2$), in which case the LHS has value equal to $\beta$.  So \mbox{$J$-circuits} of type~2 satisfy (\ref{eq:fam20}).

\item {\em Type 3.}  $x_{m-1}$ occurs last and $x_m$ next to last in $x'$.  The circuits have the form $x'=(1,\ldots,m-1,m+1)$, for which the LHS of (\ref{eq:fam20}) is
\[
(m+1)a + (m-1)b + (m-2)c + {\textstyle\frac{1}{2}}(m-3)(m-2)c \geq \beta
\]
where the inequality follows from the fact that $3a\geq 2b+c$.  So \mbox{$J$-circuits} of type~3 satisfy (\ref{eq:fam20}).  

\item {\em Type 4.}  $x_{m-1}$ occurs last but $x_m$ does not occur next to last in $x'$.  The circuits have the form $x'=(1,\ldots,m)$.  Because $b\leq c$, the LHS of (\ref{eq:fam20}) is minimized when $x_m$ occurs second from last in $x'$, in which case the LHS has value
\[
ma + (m-2)b + (m-1)c + {\textstyle\frac{1}{2}}(m-3)(m-2)c \geq \beta
\]
where the inequality follows from the fact that $2a\geq 3b$.  So \mbox{$J$-circuits} of type~4 satisfy (\ref{eq:fam20}).

\item {\em Type 5.}  Neither $x_{m-1}$ nor $x_m$ occurs last in $x'$.  The circuits have the form $x'=(1,\ldots,m)$.  Because $b\leq a\leq c$, the LHS of (\ref{eq:fam20}) is minimized when $x_{m-1}$ is second from last and $x_m$ is next to last in $x'$, in which case the LHS has value
\[
(m-2)a + (m-1)b + mc + {\textstyle\frac{1}{2}}(m-3)(m-2)c \geq \beta
\]
where the inequality follows from the fact that $c\geq 2b$.  So \mbox{$J$-circuits} of type~5 satisfy (\ref{eq:fam20}).  
\end{description}
It remains to exhibit, for each inequality (\ref{eq:fam11})--(\ref{eq:fam14}), $m$ affinely independent \mbox{$J$-circuits} that satisfy it at equality.  The scheme for doing so is very similar for (\ref{eq:fam12})--(\ref{eq:fam14}), but somewhat different for (\ref{eq:fam11}).  Beginning with (\ref{eq:fam12}), suppose for the moment that $m>3$.  We use circuits of type~1, 2, and 3, which are the only types that can satisfy (\ref{eq:fam20}) at equality:
\begin{equation}
\begin{array}{cc@{\hspace{6ex}}cccccc}
x_{m-1} & x_m     & x_{j_3} & x_{j_4} & x_{j_5} & \cdots & x_{j_{m-1}} &  x_{j_m} \vspace{.5ex} \\
\hline
m-2     & m+1     & 1       & 2       & 3       & \cdots & m-3         & m-1  \\
m-2     & m+1     & 2       & 1       & 3       & \cdots & m-3         & m-1 \\
m-2     & m+1     & 1       & 3       & 2       & \cdots & m-3         & m-1 \\
\vdots  & \vdots  & \vdots  & \vdots  & \vdots  &        & \vdots      & \vdots \\
m-2     & m+1     & 1       & 2       & 3       & \cdots & m-1         & m-3 \\
\ \\
m       & m+1     & 1       & 2       & 3       & \cdots & m-3         & m-2 \\
m+1     & m-1     & 1       & 2       & 3       & \cdots & m-3         & m-2
\end{array} \label{eq:fam25}
\end{equation}
The first $m-2$ rows are type~2 \mbox{$J$-circuits}, all of which satisfy (\ref{eq:fam20}) at equality.  The last two rows are type~1 and type~3 \mbox{$J$-circuits}, respectively, chosen as above to satisfy (\ref{eq:fam20}) at equality.  The nonsingular $(m-2)\times (m-2)$ submatrix in the upper right is obtained by swapping pairs of elements in the first row.  After suitable row operations (\ref{eq:fam25}) becomes a matrix that is triangular after rearranging columns:
\[
\begin{array}{cc@{\hspace{6ex}}cccccc}
x_{m-1} & x_m     & x_{j_3} & x_{j_4} & x_{j_5} & \cdots & x_{j_{m-1}} &  x_{j_m} \vspace{.5ex} \\
\hline
(m-2)/s & (m+1)/s & 1       & 0       & 0       & \cdots & 0           & 0  \\
(m-2)/s & (m+1)/s & 0       & 1       & 0       & \cdots & 0           & 0 \\
(m-2)/s & (m+1)/s & 0       & 0       & 1       & \cdots & 0           & 0 \\
\vdots  & \vdots  & \vdots  & \vdots  & \vdots  &        & \vdots      & \vdots \\
(m-2)/s & (m+1)/s & 0       & 0       & 0       & \cdots & 0           & 1 \\
\ \\
2+(m-2)/s
        & (m+1)/s & 0       & 0       & 0       & \cdots & 0           & 0 \\
3+2(2s-3)/(m+1)
        &  0      & 0       & 0       & 0       & \cdots & 0           & 0
\end{array} 
\]
where $s=\frac{1}{2}(m-1)(m-2)+1$ is the sum of the elements in an arbitrary row of the $(m-2)\times (m-2)$ submatrix.  Because each element on the diagonal is nonzero, the entire matrix is nonsingular, and the rows are affinely independent.  When $m=3$, we use instead the affinely independent \mbox{$J$-circuits} $(3,4,1)$, $(1,4,2)$, and $(4,2,1)$, which again are of types~1, 2 and 3 and satisfy (\ref{eq:fam12}) at equality.

Affinely independent \mbox{$J$-circuits} of types 2, 4 and 5 can be similarly exhibited for (\ref{eq:fam13}), and circuits of types 2, 3 and 4 for (\ref{eq:fam14}).  Affinely independent \mbox{$J$-circuits} for (\ref{eq:fam11}) are slightly different because only circuits of types~2 and 5 can satisfy (\ref{eq:fam11}) at equality.  Here we are given that $m\geq 4$.  We use the first $m-2$ circuits in (\ref{eq:fam25}) and the following two circuits of type~5:
\[
\begin{array}{cc@{\hspace{6ex}}cccccc}
x_{m-1} & x_m     & x_{j_3} & x_{j_4} & x_{j_5} & \cdots & x_{j_{m-1}} &  x_{j_m} \vspace{.5ex} \\
\hline
1       & m-1     & 2       & 3       & 4       & \cdots & m-2         & m  \\
2       & m-1     & 1       & 3       & 4       & \cdots & m-2         & m 
\end{array} 
\]
These satisfy (\ref{eq:fam11}) at equality because $x_{m-1}$ has the same coefficient as $x_1,x_2$.  An argument similar to the above shows that the \mbox{$J$-circuits} are affinely \mbox{independent}.   $\Box$
\endproof \medskip

The above theorems provide a complete description of facets that appear for all $m\geq d+2$ on levels~$d=0,1,2$.  We can verify this by exhaustive enumeration of facets for $m=d+2$ using Theorems~\ref{th:main} and~\ref{th:main2}.  That is, for each $d$ we use the greedy algorithm to generate all undominated \mbox{$J$-circuits} for $J=\{3,\ldots,d+4\}$.  We then consider the set $I$ of all inequalities (\ref{eqineq}), up to a positive scalar multiple, that are satisfied at equality by an affinely independent subset of $d+2$ undominated \mbox{$J$-circuits}.  Finally, we list the inequalities in $I$ that are satisfied by all the undominated \mbox{$J$-circuits}.  This list contains all inequalities that are facet defining for $m=d+2$, and all of them are described above.  This method can, in principle, be used to identify families of facets on levels 3 and higher, although for each family one must prove that it is facet defining for all $m\geq d+2$, as is done above.

\section{Separation Algorithms}
\label{separation}

There are polynomial-time separation algorithms for all of the classes of facets described in the previous two sections.  Each algorithm identifies a separating facet whenever one exists.

The separation problem is to identify a facet that separates a given solution value $\bar{x}$ of $x=(x_1, \ldots, x_n)$ from the hamiltonian circuit polytope; that is, to find a facet-defining inequality $ax\geq\alpha$ that is violated by $x=\bar{x}$.  Consider first the family (\ref{eq:permfacet}) of permutation facets.  Let $j_1, \ldots, j_{n-2}$ be an ordering of the indices $3, \ldots, n$ such that $\bar{x}_{j_1}\leq\cdots\leq \bar{x}_{j_{n-2}}$.  Then for $m=1, \ldots, n$, check whether  
\begin{equation}
\sum_{i=1}^m x_{j_i} \geq {\textstyle\frac{1}{2}} m(m+1)
\label{eq:sep1}
\end{equation}
is violated by setting $(x_{j_1}, \ldots, x_{j_m})=(\bar{x}_{j_1}, \ldots, \bar{x}_{j_m})$.  Continue until (\ref{eq:sep1}) is violated, at which point a separating facet is discovered.  The procedure has worst-case running time of $\mathcal{O}(n\log n)$, the time required to sort $n$ values.

This procedure identifies a separating permutation facet in the family (\ref{eq:permfacet}) if one exists.  To see this, suppose $\sum_{j\in J'} x_j\geq \frac{1}{2}m(m+1)$ is a separating permutation facet, where $m=|J'|$ and $1,2\not\in J'$. Then because $\bar{x}_{j_1}, \ldots, \bar{x}_{j_m}$ are the $m$ smallest values among $\bar{x}_{j_1}, \ldots, \bar{x}_{j_{n-2}}$, we have 
\[
\sum_{i=1}^m \bar{x}_{j_i} \leq \sum_{j\in J'} \bar{x}_j < {\textstyle\frac{1}{2}}m(m+1)
\]
Thus (\ref{eq:sep1}) is also separating.  

Separation requires only $\mathcal{O}(n)$ time for the two-term facets (\ref{eq:2term1})--(\ref{eq:2term5}).  A separating facet of the form (\ref{eq:2term1}) can be found, if one exists, by checking whether (\ref{eq:2term1}) is violated by setting $(x_i,x_j)=(\bar{x}_{j_1},\bar{x}_{j_2})$, where $\bar{x}_{j_1}$ and $\bar{x}_{j_2}$ are the two smallest values among $\bar{x}_1, \ldots, \bar{x}_n$.  If so, then (\ref{eq:2term1}) is separating with $(i,j)=(j_1,j_2)$.  Facets (\ref{eq:2term2})--(\ref{eq:2term5}) can be separated by enumerating at most $n$ values of the index $i$.

Level~0 facets, level~1 facets of the form (\ref{eq:fam2a}), and level~2 facets of the form (\ref{eq:fam10}) can be separated with the algorithms just described.  A single initial sort of the values $\bar{x}_1, \ldots, \bar{x}_n$ provides the basis for separating all other facets on levels~1 and~2.  For any fixed $m\geq 2$, we can find a separating level~1 facet of the form (\ref{eq:fam2}) as follows, if one exists.  Let $\bar{x}_{j_2}, \ldots, \bar{x}_{j_m}$ be the $m-1$ smallest values in $\{\bar{x}_{m+1}, \ldots, \bar{x}_n\}$.  These values can be identified in $\mathcal{O}(n)$ time by looking through the sorted elements of $\{\bar{x}_1, \ldots, \bar{x}_n\}$ and selecting the first $m-1$ elements $\bar{x}_j$ with $j> m$.  Now check whether (\ref{eq:fam2}) is violated by setting $(x_m,x_{j_1}, \ldots, x_{j_m})$ equal to $(\bar{x}_m,\bar{x}_{j_1}, \ldots, \bar{x}_{j_m})$.  If so, then (\ref{eq:fam2}) is separating.  It can be shown as above that this procedure finds a separating facet for any fixed $m$ if one exists.  We use a similar procedure for the level~2 facets (\ref{eq:fam11})--(\ref{eq:fam14}).  Thus for each $m$, we can identify a separating level~1 and level~2 facet of each type in $\mathcal{O}(n)$ time, if one exists.  By enumerating $\mathcal{O}(n)$ values of $m$, we can execute the entire separation algorithm in time $\mathcal{O}(n\log n+n^2)=\mathcal{O}(n^2)$.

As an illustration, consider circuit$(x_1,\ldots,x_7)$ with each $D_j=\{1, \ldots, 7\}$.  Suppose that $(\bar{x}_1, \ldots, \bar{x}_7)=(7,2.6,1,6.25,7,2.2,1.95)$.  This point belongs to the affine hull described by (\ref{eq:2termaffine}), but it is infeasible if only because it does not consist of
values in the domain.  The following separating cuts are identified by the above algorithms:
\begin{eqnarray}
&& x_3+x_7\geq 3 \vspace{.5ex} \label{eq:sep11} \\
&& x_2 + 2x_3 \geq 5 \vspace{.5ex} \label{eq:sep12} \\
&& x_3+2x_6+2x_7 \geq 10 \vspace{.5ex} \label{eq:sep13} \\
&& 2x_3 + x_4 + 2x_6 + 2x_7 \geq 17 \vspace{.5ex} \label{eq:sep14} \\
&& 2x_3 + x_4 + 4x_6 + 4x_7 \geq 25 \vspace{.5ex} \label{eq:sep15} \\
&& 3x_2 + 2x_3 + 4x_7 \geq 19 \vspace{.5ex} \label{eq:sep16} \\
&& 3x_2 + 2x_3 + 5x_7 \geq 21 \label{eq:sep17}
\end{eqnarray}
Here, (\ref{eq:sep11}) is a permutation facet as well as a 2-term facet, (\ref{eq:sep12}) is a level~1 facet as well as a 2-term facet, (\ref{eq:sep13}) is a level~1 facet, and (\ref{eq:sep14})--(\ref{eq:sep17}) are level~2 facets of the form (\ref{eq:fam11})--(\ref{eq:fam14}), respectively.

\section{Conclusions and Future Research}

We studied the structure of the hamiltonian circuit polytope by establishing its dimension, developing tools for the identification of facets, and using these tools to derive several families of facets.  The tools include necessary and sufficient conditions for an inequality with at most $n-4$ variables to be facet defining, stated in terms of undominated circuits, and a greedy algorithm for generating undominated circuits, for which we proved completeness.  We used a novel approach to identifying families of facet-defining inequalities, based on the structure of variable indices rather than on structured subgraphs.  Finally, we described a \mbox{hier}archy of facets of increasing combinatorial complexity and derived all facets on the first three levels.  We also presented complete polynomial-time separation algorithms for all facets described here.

%

\section*{Appendix.  Proof of Theorem~\ref{th:greedycomplete}}

To prove Theorem~\ref{th:greedycomplete}, we first define for any given circuit $\bar{x}$ an {\em implied ordering} with respect to $(J_+,J_-)$.  The proof will show that if $\bar{x}$ is undominated with respect to $(J_+,J_-)$, then a \mbox{$J$-circuit} that is greedily constructed according to the implied ordering is identical to $\bar{x}(J)$.

For a given $J$-circuit $\bar{x}(J)$, and partition $(J_+,J_-)$, let $J_+=\{i_1, \ldots,i_p\}$ where $\bar{x}_{i_1} < \cdots < \bar{x}_{i_p}$, and let $J_-=\{j_1, \ldots, j_q\}$ where $\bar{x}_{j_1} > \cdots > \bar{x}_{j_q}$.

The implied ordering will be $k_1, \ldots, k_m$.  As we construct the ordering, we construct a \mbox{$J$-circuit} $y(J)$ that is greedy with respect to the ordering.  The basic idea is that at each step $\ell$ of the procedure, we assign the greedy value to $y_{i_r}$ for the next $i_r\in J_+$ (if any remain) and let $k_{\ell}=i_r$, provided this assigns $y_{i_r}$ the same value as $\bar{x}_{i_r}$.  Otherwise, we assign the greedy value to $y_{j_s}$ for the next $j_s\in J_-$ and let $k_{\ell}=j_s$.  If no indices $j_s$ remain in $J_-$, we assign the greedy value to $y_{i_r}$ regardless of whether it agrees with $\bar{x}_{i_r}$.  The precise algorithm appears in Fig.~\ref{fig:impliedorder}.

As an example, suppose $\bar{x}=(v_2,v_3,v_4,v_7,v_6,v_1,v_5)$, $J_+=\{1,3,6,7\}$, and $J_-=\{4,5\}$.  Thus $\bar{x}(J)=(\bar{x}_1,\bar{x}_3,\bar{x}_4,\bar{x}_5,\bar{x}_6,\bar{x}_7)=(v_2,v_4,v_7,v_6,v_1,v_5)$.  Based on the values in $\bar{x}(J)$, we order the contents of $J_+$ so that $J_+=\{i_1, \ldots, i_4\}=\{6,1,3,7\}$.  Similarly, $J_-=\{j_1,j_2\}=\{4,5\}$.  The progress of the algorithm appears in \mbox{Table~\ref{ta:impliedorder}}.  Note that when $\ell=4$, we first consider assigning $v_{\min}$ to $y_{i_r}$.  But this results in $y_7=v_3$, which deviates from $\bar{x}$ because $\bar{x}_7=v_5$.  We therefore assign $v_{\max}$ to $y_{j_s}$, which yields $y_4=v_7$.  When $\ell=5$, we again consider assigning $v_{\min}$ to $y_{i_r}$, but because $v_{\min}$ has changed, we now obtain an assignment $y_7=v_5$ that agrees with $\bar{x}$.  When $\ell = 6$, the indices in $J_+$ are exhausted, and we therefore assign $v_{\min}$ to $y_{j_s}$, so that $y_5=v_6$.  The resulting $y(J)$ is identical to $\bar{x}(J)$, and the implied ordering is $(k_1, \ldots, k_6)=(6,1,3,4,5,7)$.

\bigskip
\noindent
{\bf Proof of Theorem~\ref{th:greedycomplete}.} Let $\bar{x}(J)$ be a $J$-circuit that is undominated with respect to $(J_+,J_-)$.  Let $J_+=\{i_1, \ldots,i_p\}$ where $\bar{x}_{i_1} < \cdots < \bar{x}_{i_p}$, and let $J_-=\{j_1, \ldots, j_q\}$ where $\bar{x}_{j_1} > \cdots > \bar{x}_{j_q}$.

Let $k_1, \ldots, k_m$ be the implied ordering for $\bar{x}$ with respect to $(J_+,J_-)$ as computed above, and let $(y_{k_1}, \ldots, y_{k_m})$ be the
greedy solution with respect to this ordering. We claim that $\bar{x}_{k_{\ell}}=y_{k_{\ell}}$ for $\ell=1,
\ldots, m$, which suffices to prove the theorem.  Supposing to the
contrary, let $\bar{\ell}$ be the smallest index for which
$\bar{x}_{k_{\bar{\ell}}}\neq y_{k_{\bar{\ell}}}$.  Clearly
$\bar{x}_{k_{\bar{\ell}}}\prec y_{k_{\bar{\ell}}}$ is
inconsistent with the greedy choice, because
$\bar{x}_{k_{\bar{\ell}}}$ is available when
$y_{k_{\bar{\ell}}}$ is assigned a value. Thus
we have $\bar{x}_{k_{\bar{\ell}}} \succ y_{k_{\bar{\ell}}}$

\begin{figure}[t]
\centering
\fbox{
\parbox[c]{6in}{
\begin{tabbing}
xxx \= xxx \= xxx \= xxx \= \kill
Let $V = \{v_1, \ldots, v_n\}$. \\
Let $J_+ = \{i_1, \ldots, i_p\}$ where $\bar{x}_{i_1} < \cdots < \bar{x}_{i_p}$. \\
Let $J_- = \{j_1, \ldots, j_q\}$ where $\bar{x}_{j_1} > \cdots > \bar{x}_{j_q}$. \\
Let $r=1$ and $s=1$. \\
For $\ell = 1, \ldots, m$: \\
\> Let $v_{\min}$ be the smallest value in $V$ such that setting $y_{i_r}=v_{\min}$ \\
\> \> creates no cycle with the elements of $y$ assigned so far. \\
\> Let $v_{\max}$ be the largest value in $V$ such that setting $y_{j_s}=v_{\max}$ \\
\> \> creates no cycle with the elements of $y$ assigned so far. \\
\> If $r\leq p$ and ($\bar{x}_{i_r}=v_{\min}$ or $s>q$) then \\
\> \> Let $k_{\ell}=i_r$, $y_{i_r}=v_{\min}$, and $r=r+1$. \\
\> \> Remove $v_{\min}$ from $V$. \\
\> Else \\
\> \> Let $k_{\ell}=j_s$, $y_{j_s}=v_{\max}$, and $s=s+1$. \\
\> \> Remove $v_{\max}$ from $V$.
\end{tabbing}
}
}
\vspace{0ex} \caption{Algorithm for generating an implied ordering $k_1, \ldots, k_m$ for $J$-circuit $\bar{x}(J)$ with respect to $(J_+,J_-)$, where $m=|J|$.  The resulting \mbox{$J$-circuit} $y(J)$ is greedily constructed with respect to the ordering $k_1, \ldots, k_m$ and $(J_+,J_-)$.  The algorithm is used to help prove Theorem~\ref{th:greedycomplete}, not to identify undominated \mbox{$J$-circuits} or construct facets.} \label{fig:impliedorder}
\vspace{2ex}
\end{figure}

\begin{table}
\caption{Computation of the implied ordering for $\bar{x}=(v_2,v_3,v_4,v_7,v_6,v_1,v_5)$, where $J_+=\{1,3,6,7\}$ and $J_-=\{4,5\}$ (indicated by the the signs above $\bar{x}$).}
\label{ta:impliedorder}
\vspace{1ex}
\begin{center}
$
\begin{array}{c@{\hspace{3ex}}c@{\hspace{2ex}}c@{\hspace{3ex}}c@{\hspace{1ex}}c@{\hspace{2ex}}cc@{\hspace{5ex}}ccccccc@{\hspace{3ex}}c}
     &   &   &     &     &          &          & +   &     & +   & -   & -   & +   & +   &  \vspace{-.5ex} \\
     &   &   &     &     &          & \hspace{4ex} \bar{x}= \hspace{-4ex}
                                               & v_2 & v_3 & v_4 & v_7 & v_6 & v_1 & v_5 & \\
\ell & r & s & i_r & j_s & v_{\min} & v_{\max} & y_1 & y_2 & y_3 & y_4 & y_5 & y_6 & y_7 & k_{\ell} \\
\hline
1    & 1 & 1 & 6   & 4   & v_1      & v_7      &     &     &     &     &     & v_1 &     & 6 \\
2    & 2 & 1 & 1   & 4   & v_2      & v_7      & v_2 &     &     &     &     & v_1 &     & 1 \\
3    & 3 & 1 & 3   & 4   & v_4      & v_7      & v_2 &     & v_4 &     &     & v_1 &     & 3 \\
4    & 4 & 1 & 7   & 4   & v_3      & v_7      & v_2 &     & v_4 & v_7 &     & v_1 &     & 4 \\
5    & 4 & 2 & 7   & 5   & v_5      & v_6      & v_2 &     & v_4 & v_7 &     & v_1 & v_5 & 5 \\
6    & 5 & 2 &     & 5   &          & v_6      & v_2 &     & v_4 & v_7 & v_6 & v_1 & v_5 & 7 \\
\hline
\end{array}
$
\end{center}
\vspace{4ex}
\end{table}

By hypothesis, $\bar{x}$ is undominated with respect to $(J_+\cup
J_-)$.  We therefore have $\bar{x}_{k_{\ell}} \prec
y_{k_{\ell}}$ for some $\ell\in\{\bar{\ell}+1, \ldots, m\}$.
Let $\hat{\ell}$ be the smallest such index.  Then there are two
cases: (1) $k_{\bar{\ell}}$ and $k_{\hat{\ell}}$ are both in $J_+$
or both in $J_-$, or (2) they are in different sets.
\bigskip

Case 1: $k_{\bar{\ell}}$ and $k_{\hat{\ell}}$ are both in $J_+$ or
both in $J_-$.  We will suppose that both are in $J_+$.  The
argument is similar if both are in $J_-$.

Let $t$ be the index such that $i_t=k_{\bar{\ell}}$, and $u$ the
index such that $i_u=k_{\hat{\ell}}$.  Then
$\bar{x}_{i_t}>y_{i_t}$ because $\bar{x}_{i_t}\succ
y_{i_t}$ and $i_t\in J_+$.  Let $t'$ be the largest index in
$\{t, \ldots, u-1\}$ such that $\bar{x}_{i_{t'}} >
y_{i_{t'}}$.  We know that $t'$ exists because $\bar{x}_{i_t}
> y_{i_t}$.  Thus we have two sequences of values related as
follows:
\[
\begin{array}{c@{\;}c@{\;}c@{\;}c@{\;}c@{\;}c@{\;}c@{\;}c@{\;}c@{\;}c@{\;}c@{\;}c@{\;}c@{\;}c@{\;}c@{\;}c@{\;}c@{\;}c@{\;}c}
\bar{x}_{i_1} & < & \cdots & < & 
\bar{x}_{i_{t-1}} & < & \bar{x}_{i_t} & < 
& \cdots & < & \bar{x}_{i_{t'-1}} & < & 
\bar{x}_{i_{t'}} & < & \cdots & < & 
\bar{x}_{i_{u-1}} & < & \bar{x}_{i_u} \\
$\rotatebox[origin=c]{270}{$=$}$ &   &        &   & 
$\rotatebox[origin=c]{270}{$=$}$ &   & $\rotatebox[origin=c]{270}{$>$}$ &   & 
  &   & $\rotatebox[origin=c]{270}{$\geq$}$ &   &  
$\rotatebox[origin=c]{270}{$>$}$ &   &    &   & 
$\rotatebox[origin=c]{270}{$\geq$}$ &      & $\rotatebox[origin=c]{270}{$<$}$  \\
y_{i_1} &   & \cdots &   & y_{i_{t-1}} &   &
y_{i_t} &   & \cdots &   & y_{i_{t'-1}} &   &
y_{i_{t'}} &   & \cdots &   & y_{i_{u-1}} &    &
y_{i_u}
\end{array}
\]

We first show that value $\bar{x}_{i_u}$ has not yet been assigned
in the greedy algorithm when $y_{i_u}$ is assigned a value.  That is, we show that $\bar{x}_{i_u}\not\in
\{y_{i_1}, \ldots, y_{i_{u-1}}\}$ and
$\bar{x}_{i_u}\not\in \{y_{j_1}, \ldots, y_{j_{u'}}\}$.
To see that $\bar{x}_{i_u}\not\in \{y_{i_1}, \ldots,
y_{i_{u-1}}\}$, suppose to the contrary that
$\bar{x}_{i_u}=y_{i_{w}}$ for some $w\in \{1, \ldots, u-1\}$.
This is impossible, because $\bar{x}_{i_u}>\bar{x}_{i_w}\geq
y_{i_w}$.  Also $\bar{x}_{i_u}\not\in \{y_{j_1}, \ldots,
y_{j_{u'}}\}$, because assigning value $\bar{x}_{i_u}$ to
$y_{j_w}$ for some $w\in \{1, \ldots, u'\}$ contradicts the greedy
construction of $y$, due to the fact that value
$y_{i_u}$ was available at that time and is a superior choice.

We next show that value $\bar{x}_{i_{t'}}$ has not yet been assigned
in the greedy algorithm when $y_{i_u}$ is assigned a value.  That is, we show that \mbox{$\bar{x}_{i_{t'}}\not\in
\{y_{i_1}, \ldots, y_{i_{u-1}}\}$} and
$\bar{x}_{i_{t'}}\not\in \{y_{j_1}, \ldots,
y_{j_{u'}}\}$.  To begin with, we have that
$\bar{x}_{i_{t'}}\not\in \{y_{i_1}, \ldots,
y_{i_{t'-1}}\}$, by virtue of the same reasoning just applied
to $\bar{x}_{i_u}$.  Also $\bar{x}_{i_{t'}}\neq y_{i_{t'}}$,
since by hypothesis $\bar{x}_{i_{t'}} > y_{i_{t'}}$.  To show
that $\bar{x}_{i_{t'}}\not\in \{y_{i_{t'+1}}, \ldots,
y_{i_{u-1}}\}$, suppose to the contrary that $\bar{x}_{i_{t'}}
= y_{i_w}$ for some $w\in \{t'+1, \ldots, u-1\}$.  Then since
$\bar{x}_{i_{t'}} < \bar{x}_{i_w}$, we must have $\bar{x}_{i_w} >
y_{i_w}$.  But this contradicts the definition of $t'$ ($< w$)
as the largest index in $\{1, \ldots, u-1\}$ such that
$\bar{x}_{i_{t'}} > y_{i_{t'}}$.  Thus $\bar{x}_{i_{t'}} \neq
y_{i_w}$.  Finally, $\bar{x}_{i_{t'}}\not\in \{y_{j_1},
\ldots, y_{j_{u'}}\}$ because assigning value
$\bar{x}_{i_{t'}}$ to $y_{j_w}$ for some $w\in \{1, \ldots, u'\}$
contradicts the greedy construction of $y$, due to the fact
that $y_{i_u}$ was available at the time and $y_{i_u} >
\bar{x}_{i_u} > \bar{x}_{i_{t'}}$.

Because $\bar{x}_{i_u} < y_{i_u}$ and value $\bar{x}_{i_u}$ has
not yet been assigned, setting \mbox{$y_{i_u} = \bar{x}_{i_u}$} must create
a cycle in $y$, because otherwise setting $y_{i_u} = \bar{x}_{i_u}$
would have been the greedy choice.  Also, setting
$y_{i_u}=\bar{x}_{i_{t'}}$ was not the greedy choice because
$y_{i_u} > \bar{x}_{i_u} > \bar{x}_{i_{t'}}$.  Thus setting
$y_{i_u}=\bar{x}_{i_{t'}}$ must likewise create a cycle in
$y$, because $\bar{x}_{i_{t'}}$ has not yet been assigned. Now
define $G_{y(J)}$ as before and consider the maximal subchain
in $G_{y(J)}$ that contains $y_{i_u}$.  Let the segment
of the subchain up to $y_{i_u}$ be
\[
v_z \rightarrow \cdots \rightarrow v_{i_u} \rightarrow y_{i_u}
\]
Because setting $y_{i_u}=\bar{x}_{i_u}$ creates a cycle in
$y$, we must have $\bar{x}_{i_u} = v_z$.  Similarly,
because setting $y_{i_u}=\bar{x}_{i_{t'}}$ creates a cycle in
$y$, we must have $\bar{x}_{i_{t'}} = v_z$.  This implies
$\bar{x}_{i_u}=\bar{x}_{i_{t'}}$, which is impossible because
$\bar{x}_{i_u}>\bar{x}_{i_{t'}}$.
\bigskip

Case 2: $k_{\bar{\ell}}\in J_+$ and $k_{\hat{\ell}}\in J_-$, or
$k_{\bar{\ell}}\in J_-$ and $k_{\hat{\ell}}\in J_+$.  We can rule
out the latter subcase immediately, because $k_{\bar{\ell}}$ can be
in $J_-$ only if $r>p$ when $y_{k_{\bar{\ell}}}$ is assigned
a value.  This means $k_{\hat{\ell}}$ must be in
$J_-$ as well, because $y_{k_{\hat{\ell}}}$ is assigned a value after
$y_{k_{\bar{\ell}}}$ is assigned a value, and the situation reverts to Case 1.  We
therefore suppose $k_{\bar{\ell}}\in J_+$ and $k_{\hat{\ell}}\in
J_-$.

Let $t$ be the index such that $i_t=k_{\bar{\ell}}$, and $u$ the
index such that $j_u=k_{\hat{\ell}}$.  Again $\bar{x}_{i_t} >
y_{i_t}$ because $\bar{x}_{i_t} \succ y_{i_t}$ and
$j_t\in J_+$.  Thus, at the time value $y_{i_t}$ was assigned
a value, we had $\bar{x}_{j_s}<v_{\max}$ for the current value
of $s$.  So we have two sequences of values related as follows:
\begin{equation}
\begin{array}{c@{\;}c@{\;}c@{\;}c@{\;}c@{\;}c@{\;}c@{\;}c@{\;}c@{\;}c@{\;}c@{\;}c}
\bar{x}_{j_1} & > & \cdots & > & 
\bar{x}_{j_{s-1}} & > & \bar{x}_{j_{s}} & > & 
\cdots & \bar{x}_{j_{u-1}} & > & \bar{x}_{j_u} \\
$\rotatebox[origin=c]{270}{$=$}$ &   &   &   & 
$\rotatebox[origin=c]{270}{$=$}$ &   & $\rotatebox[origin=c]{270}{$\leq$}$ &   &        & 
$\rotatebox[origin=c]{270}{$\leq$}$ &   & $\rotatebox[origin=c]{270}{$>$}$ \\
y_{j_1} &   & \cdots &   & y_{j_{s-1}} &   &
y_{j_{s}} &   & \cdots & y_{j_{u-1}} &   & y_{j_u}
\end{array} \label{eq:J+}
\end{equation}
where $v_{\max}> \bar{x}_{j_s}$.  Let $t'$ be the largest index for
which $y_{i_{t'}}$ has been assigned a value at the time
$y_{j_u}$ is assigned a value.  We have two sequences of
values related as follows:
\[
\begin{array}{c@{\;}c@{\;}c@{\;}c@{\;}c@{\;}c@{\;}c@{\;}c@{\;}c@{\;}c@{\;}c@{\;}c@{\;}c@{\;}c@{\;}c@{\;}c@{\;}c@{\;}c@{\;}c}
\bar{x}_{i_1} & < & \cdots & < & 
\bar{x}_{i_{t-1}} & < & \bar{x}_{i_t} & < & 
\cdots & < & \bar{x}_{i_{t'}} \\
$\rotatebox[origin=c]{270}{$=$}$ &   &        &   & 
$\rotatebox[origin=c]{270}{$=$}$ &   & $\rotatebox[origin=c]{270}{$>$}$ &   &        &   & \\
y_{i_1} &   & \cdots &   & y_{i_{t-1}} &   &
y_{i_t} &   & \cdots &   & y_{i_{t'}}
\end{array}
\]
We first show that a cycle must be created if value $\bar{x}_{j_u}$
is assigned to $y_{j_u}$.  Because
$y_{j_u}<\bar{x}_{j_u}$, it suffices to show that value
$\bar{x}_{j_u}$ has not yet been assigned in the greedy algorithm
when $y_{j_u}$ is assigned a value.  That is, we show
that $\bar{x}_{j_u}\not\in \{y_{j_1}, \ldots,
y_{j_{u-1}}\}$ and $\bar{x}_{j_u}\not\in \{y_{i_1},
\ldots, y_{i_{t'}}\}$.  If $\bar{x}_{j_u}=y_{j_{w}}$ for
some $w\in \mbox{$\{1, \ldots, u-1\}$}$, then
$\bar{x}_{j_u}<\bar{x}_{j_w}\leq y_{j_w}$, which is
impossible.  Thus $\bar{x}_{j_u}\not\in \{y_{j_1}, \ldots,
y_{j_{u-1}}\}$.  Also $\bar{x}_{j_u}\not\in \{y_{i_1},
\ldots, y_{i_{t'}}\}$, because assigning value $\bar{x}_{j_u}$
to $y_{i_w}$ for some $w\in \{1, \ldots, t'\}$ contradicts the
greedy construction of $y$, due to the fact that value
$y_{j_u}$ was available at that time and is a superior choice.

We next show that a cycle must be created if value $v_{\max}$ is assigned to $y_{j_u}$.  Note that
$v_{\max}\not\in \{y_{i_1}, \ldots, y_{i_{t'}}\}$,
because assigning value $v_{\max}$ to $y_{i_w}$ for some $w\in \{1,
\ldots, t'\}$ contradicts the greedy construction of $y$, due
to the fact that value $y_{j_u}$ was available at that time
and is a superior choice because
$v_{\max}>\bar{x}_{j_s}>\bar{x}_{j_u}$.  Now suppose, contrary to
the claim, that assigning $v_{\max}$ to $y_{j_u}$ does not create a
cycle.  Then since $v_{\max}>y_{j_u}$, the value $v_{\max}$
must have already been assigned in the greedy algorithm at the time
$y_{j_u}$ is assigned a value.  This implies $v_{\max}\in
\{y_{j_s}, \ldots, y_{j_{u-1}}\}$.  But in this case we
must have $y_{j_s}=v_{\max}$, because assigning $v_{\max}$ to
$y_{j_s}$ does not create a cycle and, by definition, is the most
attractive choice at the time.  Thus (\ref{eq:J+}) becomes
\[
\begin{array}{c@{\;}c@{\;}c@{\;}c@{\;}c@{\;}c@{\;}c@{\;}c@{\;}c@{\;}c@{\;}c@{\;}c@{\;}c@{\;}c@{\;}c@{\;}c@{\;}c@{\;}c@{\;}c}
\bar{x}_{j_1} & > & \cdots & > & \bar{x}_{j_{s-1}} & > & \bar{x}_{j_s} & > & \cdots & > & \bar{x}_{j_{s'-1}} & > & \bar{x}_{j_{s'}} & > & \cdots & > & \bar{x}_{j_{u-1}} & > & \bar{x}_{j_u} \\
$\rotatebox[origin=c]{270}{$=$}$ &   &   &   & $\rotatebox[origin=c]{270}{$=$}$ &   & $\rotatebox[origin=c]{270}{$<$}$ &   &  &   & $\rotatebox[origin=c]{270}{$\leq$}$ &   & $\rotatebox[origin=c]{270}{$<$}$ &   &   &   & $\rotatebox[origin=c]{270}{$\geq$}$ &   & $\rotatebox[origin=c]{270}{$<$}$  \\
y_{j_1} &   & \cdots &   & y_{j_{s-1}} &   &
y_{j_s} &   & \cdots &   & y_{j_{s'-1}} &   &
y_{j_{s'}} &   & \cdots &   & y_{j_{u-1}} &   &
y_{j_u}
\end{array}
\]
where $y_{j_s} = v_{\max}$ and where $s'$ is the largest index
in $\{s, \ldots, u-1\}$ such that
$y_{j_{s'}}<\bar{x}_{j_{s'}}$.  Now we can argue as in Case 1
that assigning $\bar{x}_{j_u}$ to $y_{j_u}$ creates a cycle, and
assigning $\bar{x}_{j_{s'}}$ to $y_{j_u}$ creates a cycle, which
implies $\bar{x}_{j_{s'}}=\bar{x}_{j_u}$, a contradiction because
$\bar{x}_{j_{s'}}>\bar{x}_{j_u}$.  We conclude that assigning
$v_{\max}$ to $y_{j_u}$ creates a cycle.

Having shown that assigning $\bar{x}_{j_u}$ to $y_{j_u}$ creates a
cycle, and assigning $v_{\max}$ to $y_{j_u}$ creates a cycle, we
derive as in Case 1 that $v_{\max}=\bar{x}_{j_u}$, a contradiction
because $v_{\max}\geq \bar{x}_{j_s}> \bar{x}_{j_u}$.  The theorem
follows.  
$\Box$ \medskip

%
%





\end{document}